\documentclass[psamsfonts]{conm-p-l}

\usepackage{amssymb}
\usepackage{graphicx}
\usepackage{latexsym}
\usepackage{tikz}
\usetikzlibrary{decorations.pathreplacing}

\newtheorem{lemma}{Lemma}
\newtheorem{definition}{Definition}
\newtheorem{theorem}{Theorem}
\newtheorem{proposition}{Proposition}

\newtheorem{corollary}{Corollary}
\newtheorem{example}{Example}
\newtheorem{remark}{Remark}
\newtheorem{conjecture}{Conjecture}

\def\mmset{{\mathcal B}}
\def\mmseto{\overline{\mmset}}

\def\cD{{\mathcal D}}

\def\R{\mathbb{R}}
\def\Tmax{T_{\max}}

\def\bez{\backslash}
\def\podx{\underline{x}}
\def\nadx{\overline{x}}

\def\conv{\operatorname{conv}_{\oplus}}
\def\convo{\overline{\conv}}

\def\cB{\mmset}
\def\Lin{\operatorname{Lin}}
\def\supp{\operatorname{supp}}
\def\rank{\operatorname{rank}}
\def\inter{\operatorname{int}}

\numberwithin{equation}{section}

\begin{document}

\title{Tropical convexity over max-min semiring}

\author{Viorel Nitica}
\address{Department of Mathematics, West Chester University, PA
19383, USA, and Institute of Mathematics, P.O. Box 1-764, Bucharest, Romania}
\email{vnitica@wcupa.edu}

\author{Serge\u{\i} Sergeev}
\address{University of Birmingham,
School of Mathematics, Watson Building, Edgbaston B15 2TT, UK}
\email{sergiej@gmail.com}

\thanks{Viorel Nitica was partially supported by a grant from Simons Foundation 208729. Serge\u{\i} Sergeev
is supported by EPSRC grant RRAH15735, RFBR-CNRS grant 11-0193106
and RFBR grant 12-01-00886.} 
\date{}
\maketitle

\begin{abstract} This is a survey on
an analogue of tropical convexity developed over the max-min
semiring, starting with the descriptions of max-min segments,
semispaces, hyperplanes and an account of separation and non-separation results
based on semispaces. There are some new results. 
In particular, we give new ``colorful'' extensions of the max-min
Carath\'eodory theorem. In the end of the paper, we list some
consequences of the topological Radon and Tverberg theorems (like
Helly and Centerpoint theorems), valid over a more general class of
max-T semirings, where multiplication is a triangular norm.
\end{abstract}

\section{Introduction\label{sec1}}

The max-min semiring is defined as the unit interval $\mmset=[0,1]$ with the operations
$a\oplus b:=\max(a,b)$, as addition, and $a\otimes b:=\min(a,b)$, as multiplication. The operations are idempotent,
$\max(a,a)=a=\min(a,a)$, and related to the
order:
\begin{equation}\label{first-eq-662}
\max(a,b)=b\Leftrightarrow a\leq b\Leftrightarrow \min(a,b)=a.
\end{equation}
One can naturally extend them to matrices and vectors
leading to the max-min (fuzzy) linear algebra
\cite{BCS-87,Gav-01,Gav:04}. We denote by $\mmset(d,m)$ the set
of $d\times m$ matrices with entries in $\mmset$ and by $\mmset^d$
the set of $d$-dimensional vectors with entries in $\mmset$. Both
$\mmset(d,m)$ and $\mmset^d$ have a natural structure of semimodule
over the semiring $\mmset$.

The {\bf max-min segment} between $x,y\in\mmset^d$ is defined as
\begin{equation}
\label{segm0}
\begin{aligned}
\ [x,y]_{\oplus} &= \{\alpha\otimes x\oplus \beta\otimes y\mid \,\alpha \oplus \beta =1, \alpha,\beta\in\mmset\}.\
\end{aligned}
\end{equation}

A set $C\subseteq\mmset^d$ is called {\bf max-min convex}, if it contains, with any two points
$x,y,$ the segment $[x,y]_{\oplus}$ between them. For a general subset $X\subseteq\mmset^d$,
define its {\bf convex hull} $\conv(X)$ as the smallest max-min convex set containing $X$, i.e.,
the smallest set containing $X$ and stable under taking segments~\eqref{segm0}.
As in the ordinary convexity, $\conv(X)$ is the set of
all {\em max-min convex combinations}
\begin{equation}
\label{convX}
\bigoplus_{i=1}^m \lambda_i\otimes x^i\colon m\geq 1,\ \bigoplus_{i=1}^m \lambda_i=1,
\end{equation}
of all $m$-tuples of elements $x^1,\ldots,x^m\in X$. The max-min convex hull
of a finite set of points is also called a {\em max-min convex polytope}.

A {\bf (max-min) semispace} at $x\in\mmset^d$ is defined as a maximal max-min convex set not containing $x$.
A straightforward application of Zorn's Lemma shows that if $C\subseteq \mmset^d$ is convex and $x\notin C$, then
$x$ can be separated from $C$ by a semispace. It follows that the semispaces constitute the
smallest intersectional basis of max-min convex sets. This fact is true more generally in abstract convexity.
Some new phenomena appear in max-min convexity, which further emphasize the importance of semispaces
in any convexity theory. For example, separation of a point and a convex set by hyperplanes is not always possible
in max-min convexity \cite{Nit-09}, \cite{N-Ser1}.

The max-min segments and semispaces were described, respectively,
in~\cite{NS-08I,Ser-03} and in~\cite{NS-08II}. In the present
paper, the max-min segments are introduced in
Section~\ref{s:segments}. We recall the structure of max-min
semispaces in Section~\ref{s:semispaces} together with some
immediate consequences from abstract convexity.
In~\cite{N-Ser1,N-Ser2} further progress is made in the study of
max-min convexity focusing on the role of semispaces. Being
motivated by the Hahn-Banach separation theorems in the tropical
(max-plus) convexity~\cite{Zim-77} and extensions to functional and
abstract idempotent semimodules~\cite{CGQS-05, LMS-01, Zim-81}, we
compared semispaces to max-min hyperplanes in~\cite{N-Ser1}, and
developed an interval extension of separation by semispaces
in~\cite{N-Ser2}. These results are summarized in
Section~\ref{s:separation}.
Another principal goal of this paper is to investigate classical
convexity results such as the theorems of Carath\'eodory, Helly and
Radon in the realm of max-min convexity. These results are presented
in Sections~\ref{s:carath},~\ref{s:intsep} and~\ref{s:radon-helly}
and are inspired by a paper of Gaubert and Meunier~\cite{G-Meu}, in
which similar statements can be found for the case of max-plus
convexity. The max-min Carath\'{e}odory theorem with some
``colorful'' extensions is presented in Section~\ref{s:carath}. The
strongest extension relies on what we call the internal separation
theorem, which is proved in Section~\ref{s:intsep}. In the last
section, motivated by the fuzzy algebra of~\cite{pap}, we consider a
more general class of max-T semirings, where the role of
multiplication is played by a triangular norm. We show how the
topological Radon and Tverberg theorems can be applied to obtain, in
particular, the max-min analogues of Radon, Helly, Centerpoint and
(in part) Tverberg theorems.

\section{Description of segments}\label{s:segments}

In this section we describe general segments in $\mmset^{d},$
following~\cite{NS-08I,Ser-03}, where complete proofs can be found.
Note that the description of the segments in \cite{NS-08I, Ser-03}
is done for the equivalent case where $\mmset=[-\infty, +\infty]$.

Let $x=(x_{1},...,x_{d}),$ $y=(y_{1},...,y_{d})\in\mmset^{d},$
and assume that we are in the \emph{case of comparable endpoints},
say $x\leq y$ in the natural
order of $\mmset^{d}.$
Sorting the set of all coordinates $\{x_{i},y_{i},i=1,...,d\}$
we obtain a non-decreasing sequence, denoted by $t_1,t_2,\ldots, t_{2d}$.
This sequence divides the set $\mmset$ into
$2d+1$ subintervals
$\sigma_0=[0,t_{1}],\,\sigma_1=[t_1,t_2],...,\sigma_{2d}=[t_{2d},1]$,
with consecutive subintervals having one common endpoint.

Every point $z\in [x,y]_{\oplus}$ is represented as
$z=\alpha\otimes x\oplus\beta\otimes y$, where $\alpha=1$ or $\beta=1$.
However, case $\beta=1$ yields only $z=y$, so we can assume $\alpha=1$.
Thus $z$ can be regarded as a function of one parameter $\beta$, that is,
$z(\beta)=(z_{1}(\beta),...,z_{d}(\beta))$ with $\beta\in\mmset$.
Observe that for $\beta\in\sigma_0$ we have $z(\beta)=x$ and for
$\beta\in\sigma_{2d}$ we have $z(\beta)=y$. Vectors $z(\beta)$ with
$\beta$ in any other subinterval form a conventional {\em elementary segment}.
Let us proceed with a formal account of all this.

\begin{theorem}
\label{tcomp}
Let $x,y\in\mmset^d$ and $x\leq y$.
\begin{itemize}
\item[(i)] We have
\begin{equation}
\label{e:chain}
[x,y]_{\oplus}=\bigcup_{l=1}^{2d-1} \{z(\beta)\mid\beta\in\sigma_l\},
\end{equation}
where $z(\beta)=x\oplus(\beta\otimes y)$ and $\sigma_{\ell}=[t_l,t_{l+1}]$ for
$\ell=1,\ldots,2d-1$, and $t_1,\ldots,t_{2d}$ is the nondecreasing sequence
whose elements are the coordinates $x_i,y_i$ for $i=1,\ldots,d$.
\item[(ii)] For each $\beta\in\mmset$ and $i$, let
$M(\beta)=\{i\colon x_i\leq\beta\leq y_i\}$, $H(\beta)=\{i\mid\beta\geq y_i\}$
and $L(\beta)=\{i\colon\beta\leq x_i\}$. Then
\begin{equation}
\label{zibeta}
z_i(\beta)=
\begin{cases}
\beta, &\text{if $i\in M(\beta)$},\\
x_i, &\text{if $i\in L(\beta)$},\\
y_i, &\text{if $i\in H(\beta)$},
\end{cases}
\end{equation}
and $M(\beta),L(\beta), H(\beta)$ do not change in the interior
of each interval $\sigma_{\ell}$.
\item[(iii)]  The sets $\{z(\beta)\mid\beta\in\sigma_{\ell}\}$ in~\eqref{e:chain}
are conventional closed segments in $\mmset^d$ (possibly reduced to a point),
described by~\eqref{zibeta} where $\beta\in\sigma_{\ell}$.
\end{itemize}
\end{theorem}

For \emph{incomparable endpoints} $x\not\leq y,\,y\not\leq x,$
the description can be reduced to that of segments with comparable endpoints,
by means of the following observation.

\begin{theorem}
\label{tincomp}Let $x,y\in
\mmset^d$. Then $[x,y]_{\oplus}$ is the
concatenation of two segments with comparable endpoints, namely $\lbrack x,y]_{\oplus}=[x,x\oplus y]_{\oplus}\cup [x\oplus y,y]_{\oplus}.$
\end{theorem}

All types of segments for $d=2$ are shown in the right side of Figure 1.

The left side of Figure 1 shows a diagram, where for
$x=(x_1,x_2,x_3)$ and $y=(y_1, y_2, y_3)$, the segments $[x_1,y_1],
[x_2,y_2],$ and $[x_3,y_3]$ are placed over one another, and their
arrangement induces a tiling of the horizontal axis, which shows the
possible values of the parameter $\beta$. The partition of the real
line induced by this tiling is associated with the intervals
$\sigma_l$, and the sets of {\em active indices} $i$ with
$z_i(\beta)=\beta$ associated with each $\sigma_l$ are also shown.

\begin{figure}[h]
\begin{tikzpicture}[scale=.68]
\node at (4,3) {Segments in $\mmset^2$, comparable endpoints};
\draw [line width = 1] (0,0)--(2,0)--(2,2)--(0,2)--(0,0)--(2,2);
\draw [line width = 2] (.1, .6)--(.45,.6)--(.45,1.7);
\draw [line width = 1] (3,0)--(5,0)--(5,2)--(3,2)--(3,0)--(5,2);
\draw [line width = 2] (3.1,.6)--(3.6,.6)--(4.5,1.5)--(4.5,1.9);
\draw [line width = 1] (6,0)--(8,0)--(8,2)--(6,2)--(6,0)--(8,2);
\draw [line width = 2] (6.1,.6)--(6.6,.6)--(7.5,1.5)--(7.9,1.5);
\draw [line width = 1] (0,-3)--(2,-3)--(2,-1)--(0,-1)--(0,-3)--(2,-1);
\draw [line width = 2] (.5,-2.9)--(.5,-2.6)--(1.5,-2.6);
\draw [line width = 1] (3,-3)--(5,-3)--(5,-1)--(3,-1)--(3,-3)--(5,-1);
\draw [line width = 2] (3.5,-2.9)--(3.5,-2.5)--(4.5,-1.5)--(4.5,-1.1);
\draw [line width = 1] (6,-3)--(8,-3)--(8,-1)--(6,-1)--(6,-3)--(8,-1);
\draw [line width = 2] (6.5,-2.9)--(6.5,-2.5)--(7.5,-1.5)--(7.8,-1.5);
\node at (4,-4) {Segment in $\mmset^2$, incomparable endpoints};
\draw [line width = 1] (3,-7)--(5,-7)--(5,-5)--(3,-5)--(3,-7)--(5,-5);
\draw [line width = 2] (3.4,-5.5)--(4.6,-5.5)--(4.6,-6.7);

\draw [line width = 1] (-10,-6)--(-2,-6);
\draw [line width = 2] (-9,-2)--(-5,-2);
\draw [line width = 2] (-8,-1)--(-3,-1);
\draw [line width = 2] (-7,0)--(-4,0);
\draw [line width = 1] [dotted] (-9,-6)--(-9,1);
\draw [line width = 1] [dotted] (-8,-6)--(-8,1);
\draw [line width = 1] [dotted] (-7,-6)--(-7,1);
\draw [line width = 1] [dotted] (-5,-6)--(-5,1);
\draw [line width = 1] [dotted] (-4,-6)--(-4,1);
\draw [line width = 1] [dotted] (-3,-6)--(-3,1);
\node at (-9,-6.5) {$t_1$};
\node at (-8,-6.5) {$t_2$};
\node at (-7,-6.5) {$t_3$};
\node at (-5,-6.5) {$t_4$};
\node at (-4,-6.5) {$t_5$};
\node at (-3,-6.5) {$t_6$};
\node at (-1.8,-6.5) {$\beta$};

\node at (-9.5,-5.5) {\tiny{$\emptyset$}};
\node at (-8.5,-5.5) {\tiny{$\{1\}$}};
\node at (-7.5,-5.5) {\tiny{$\{1,2\}$}};
\node at (-6,-5.5) {\tiny{$\{1,2,3\}$}};
\node at (-4.5,-5.5) {\tiny{$\{2,3\}$}};
\node at (-3.5,-5.5) {\tiny{$\{2\}$}};
\node at (-2.5,-5.5) {\tiny{$\emptyset$}};
\node at (-1.5,-5.5) {\tiny{$M(\beta)$}};

\node at (-9.5,-4.5) {$\sigma_0$};
\node at (-8.5,-4.5) {$\sigma_1$};
\node at (-7.5,-4.5) {$\sigma_2$};
\node at (-6,-4.5) {$\sigma_3$};
\node at (-4.5,-4.5) {$\sigma_4$};
\node at (-3.5,-4.5) {$\sigma_5$};
\node at (-2.5,-4.5) {$\sigma_6$};

\node at (-9.5,-2) {$x_1$};
\node at (-4.5,-2) {$y_1$};
\node at (-8.5,-1) {$x_2$};
\node at (-2.5,-1) {$y_2$};
\node at (-7.5,0) {$x_3$};
\node at (-3.5,0) {$y_3$};

\node at (-5.5,3) {Diagram showing intervals $\sigma_{\ell}$ and sets};
\node at (-5.5,2.3) {of coordinates moving together $M(\beta)$};
\end{tikzpicture}
\caption{Max-min segments.}
\end{figure}
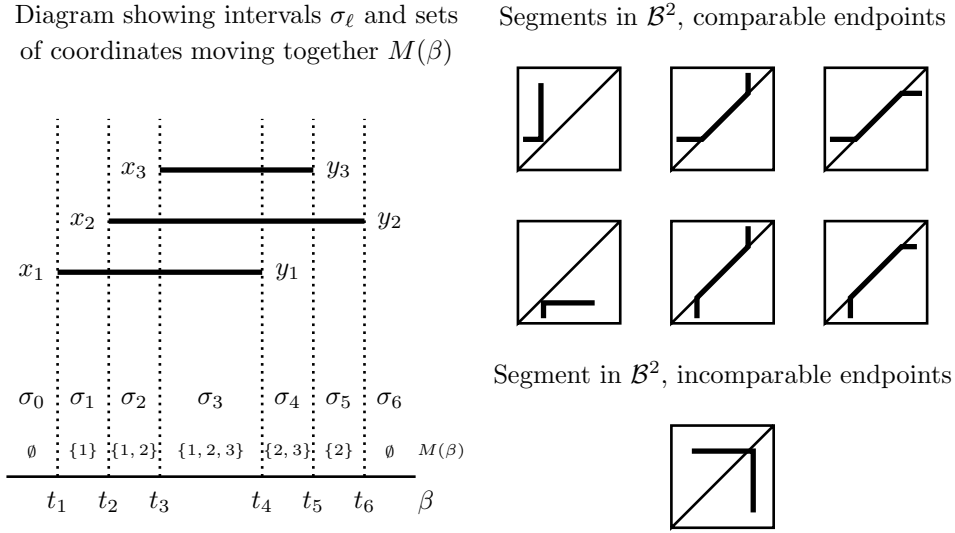
\medskip

\begin{remark}
\label{rsamerole}  We observe that, similarly
to the max-plus case (see \cite{NS-1}, Remark 4.3) in $\mmset^d$ there are elementary segments in only $2^d-1$ directions. Elementary
segments are the "building blocks" for the max-min segments in $\mmset^d,$ in the sense that every segment $[x,y]\subset \mmset^d$ is the concatenation of a
finite number of elementary subsegments (at most) $2d-1$, respectively $2d-2$, in the case of comparable, respectively incomparable, endpoints.
\end{remark}

Max-min segments allow to introduce a natural metric on $\mmset^d$ (\cite{EJN}). More precisely, one defines the distance between two points to be the Euclidean length of the max-min segment joining them.

\section{Description of semispaces}\label{s:semispaces}

For any point $x^0=(x_{1}^{0}, \dots, x_{d}^{0})\in \mmset^d$ we
define a finite family of subsets $S_0(x^0),\dots,S_d(x^0)$ in
$\mmset^d$. These subsets were shown to be semispaces in
\cite[Proposition 4.1]{NS-08II}. A point $x^0$ is called {\em
finite} if it has all coordinates different from zeros and ones.
This definition is motivated by the isomorphic version of max-min
algebra where the least element (and zero of the semiring) is
$-\infty$, and the greatest element (and unity of the semiring) is
$+\infty$.

Without loss of generality we may assume that $x^0$ is {\bf non-increasing}: $ x_{1}^{0}\geq \dots \geq x_{d}^{0}. $ Writing
this more precisely we have
\begin{equation}\label{permut5}
\begin{gathered}
x_{1}^{0}
=\dots =x_{k_{1}}^{0}>\dots >x_{k_{1}+l_{1}+1}^{0}=\dots =x_{k_{1}+l_{1}+k_{2}}^{0}>\dots \\
>x_{k_{1}+l_{1}+
k_{2}+l_{2}+1}^{0}=\dots =x_{k_{1}+l_{1}+k_{2}+l_{2}+k_{3}}^{0}>\dots\\
>x_{k_{1}+l_{1}+\dots +k_{p-1}+l_{p-1}+1}^{0}=\dots =x_{k_{1}+l_{1}+\dots +k_{p-1}+l_{p-1}+k_{p}}^{0}\\
>\dots >x^0_{k_{1}+l_{1}+\dots +k_{p}+l_{p}}(=x^0_d),
\end{gathered}
\end{equation}
where $\sum_{j=1}^p(k_j+l_j)=d$, $k_1=0$ if the sequence \eqref{permut5} starts with strict inequalities and $l_p=0$ if the sequence ends with equalities.

Let us introduce the following notations:
\begin{equation*}
\begin{split}
L_{0} &= 0,K_{1}=k_{1},L_{1}=K_{1}+l_{1}=k_{1}+l_{1},\\
K_{j} &= L_{j-1}+k_{j}=k_{1}+l_{1}+...+k_{j-1}+l_{j-1}+k_{j}\quad
(j=2,...,p), \\
L_{j} &= K_{j}+l_{j}=k_{1}+l_{1}+...+k_{j}+l_{j}\quad (j=2,...,p);
\end{split}
\end{equation*}
we observe that $l_{j}=0$ if and only if $K_{j}=L_{j}.$

We are ready to define the subsets. We need to distinguish the cases when the
sequence \eqref{permut5} ends with zeros or begin with ones, since some subsets $S_i$ become empty in that case.

\begin{definition}
\label{def:semi}
Let $x^0\in\mmset^d$ be a non-increasing vector\\
a) If $x^{0}$ has $0<x_i^0<1$ for all
$1\leq i\leq d$, then define:
\begin{equation*}
\begin{split}
S_{0}(x^0)=&\{x\in \mmset^{d}|x_{i}>x_{i}^{0}\text{ for some }1\leq i\leq
d\}, \\
S_{K_{j}+q}(x^0)=&\{x\in \mmset^{d}|x_{K_{j}+q}<x_{K_{j}+q}^{0},\text{ or }%
x_{i}>x_{i}^{0}\\
& \text{ for some }K_{j}+q+1\leq i\leq d\} (q=1,...,l_{j};j=1,...,p \text{ if }l_{j}\neq 0),\\
S_{L_{j-1}+q}(x^0)=&\{x\in \mmset^{d}|x_{L_{j-1}+q}<x_{L_{j-1}+q}^{0},\text{
or }x_{i}>x_{i}^{0}\\
& \text{ for some }K_{j}+1\leq i\leq d\} \\
& (q=1,...,k_{j};j=1,...,p\text{ if }k_{1}\neq 0,\text{ or }j=2,...,p\text{ if
}k_{1}=0).
\end{split}
\end{equation*}

b) If there exists an index $i\in \{1,...,d\}$\ such that $x_{i}^{0}=1,$ but no index $j$ such that $x_{j}^{0}=0,$ then define the subsets $S_{1},...,S_{d}$ as in part a).

c) If there exists an index $j\in \{1,...,d\}$\ such that $x_{j}^{0}=0,$ but no index $i$ such
that $x_{i}^{0}=1,$ then define the subsets $S_{0},S_{1},...,S_{\beta -1}$ as in part a), where
$
\beta :=\min \{1\leq j\leq n|\;x_{j}^{0}=0 \}.
$

d) If there exists an index $i\in \{1,...,d\}$\ such that $x_{i}^{0}=1,$
and an index $j$ such that $x_{j}^{0}=0,$ then define the subsets $S_{1},...,S_{\beta -1}$ as in part a), where
$
\beta :=\min \{1\leq j\leq n|\;x_{j}^{0}=0 \}.
$
\end{definition}

Let now $x^0\in\mmset^d$ have arbitrary order of coordinates, and let us
formally extend Definition~\ref{def:semi}. For this, consider a permutation
$\pi$ of the index set $\{1,\ldots,d\}$ such that the vector
$(x_{\pi(1)}, x_{\pi(2)},\ldots,x_{\pi(d)})$ is non-increasing. Let
$\overline{\pi}:\mmset^d\to\mmset^d$ be the invertible map of $\mmset^d$ induced by
the permutation $\pi$. Then we can define
$S_i(x^0)=\overline{\pi}^{-1} (S_j(\overline{\pi}(x^0)))$, where $j=\pi(i)$.

Further, for any $x^0\in \mmset^d$ we denote by $I(x^0)$ the set of indices $i$
such that $S_{\pi(i)} (\overline{\pi}(x^0))$ is present in Definition~\ref{def:semi}. Observe
that $I(x^0)$ consists of the indices $i$ such that $x^0_i>0$ and, possibly, $0$.

Pictures of all semispaces at a finite point for $d=2$ are shown in Figure 2.

\begin{figure}[h]
\begin{tikzpicture}[scale=1.3]
\node at (4,2.5) {Semispaces at a point with equal coordinates};

\draw [line width = 0] [fill = lightgray] (.7,0)--(2,0)--(2,2)--(0,2)--(0,.7)--(.7,.7)--(.7,0);
\draw [line width = 1]  (.7,0)--(2,0)--(2,2)--(0,2)--(0,.7);
\draw [dotted, line width = 1] (0,.7)--(0,0)--(.7,0)--(.7,.7)--(0,.7);
\draw[fill=white] (.7,.7) circle (.05);

\draw [line width = 0] [fill = lightgray] (3.7,0)--(3,0)--(3,2)--(3.7,2)--(3.7,0);
\draw [line width = 1] (3.7,0)--(3,0)--(3,2)--(3.7,2);
\draw [dotted, line width = 1] (3.7,0)--(5,0)--(5,2)--(3.7,2)--(3.7,0);
\draw[fill=white] (3.7,.7) circle (.05);

\draw [line width = 0] [fill = lightgray] (6,.7)--(6,0)--(8,0)--(8,.7)--(6,.7);
\draw [line width = 1] (6,.7)--(6,0)--(8,0)--(8,.7);
\draw [dotted, line width = 1] (6,.7)--(6,2)--(8,2)--(8,.7)--(6,.7);
\draw[fill=white] (6.7,.7) circle (.05);

\node at (4,-.5) {Semispaces at a point with unequal coordinates};

\draw [line width = 0] [fill = lightgray] (1.2,-3)--(1.2,-2.5)--(0,-2.5)--(0,-1)--(2,-1)--(2,-3)--(1.2,-3);
\draw [line width = 1] (1.2,-3)--(2,-3)--(2,-1)--(0,-1)--(0,-2.5);
\draw [dotted, line width = 1] (0,-3)--(1.2,-3)--(1.2,-2.5)--(0,-2.5)--(0,-3);
\draw[fill=white] (1.2,-2.5) circle (.05);

\draw [line width = 0] [fill = lightgray] (4.2,-3)--(4.2,-2.5)--(5,-2.5)--(5,-1)--(3,-1)--(3,-3)--(4.2,-3);
\draw [line width = 1] (3,-3)--(4.2,-3);
\draw [line width = 1] (5,-2.5)--(5,-1)--(3,-1)--(3,-3);
\draw [dotted, line width = 1] (4.2,-3)--(5,-3)--(5,-2.5)--(4.2,-2.5)--(4.2,-3);
\draw[fill=white] (4.2,-2.5) circle (.05);

\draw [line width = 0] [fill = lightgray] (6,-2.5)--(6,-3)--(8,-3)--(8,-2.5)--(6,-2.5);
\draw [line width = 1] (6,-2.5)--(6,-3)--(8,-3)--(8,-2.5);
\draw [dotted, line width = 1] (6,-2.5)--(8,-2.5)--(8,-1)--(6,-1)--(6,-2.5);
\draw[fill=white] (7.2,-2.5) circle (.05);
\end{tikzpicture}
\caption{Semispaces in $\mmset^2$ at a finite point}
\end{figure}

The following theorem is the main result in \cite{NS-08II}. See also \cite{N-Ser2}.

\begin{theorem} \label{p:semisp-conv}
For any $p\in \mmset^d$ the sets $S_i(p), i\in I(p),$ are maximal (with respect to the set inclusion)
max-min convex avoiding the point $p$.
Thus for any $p\in \mmset^d$, there exists at least one and at most $d+1$ semispaces $S_i(p), 0\le i\le d,$ at $p$.

For all $C\subseteq \mmset^d$ max-min convex and any $p\in \mmset^d
\setminus C$, there exists a semispace $S_i(p)$ such that
$C\subseteq S_i(p)$ and $p\not \in S_i(p)$.
\end{theorem}

The complement of a semispace $S_i(p)$ is denoted by $\complement
S_i(p)$. These complements are also called {\em sectors}, in analogy
with the max-plus convexity.

The lemma below follows from the abstract definition of the
semispaces and it is our main tool in extending Carath\'eodory
theorem and its colorful versions to the max-min setup. As only a
finite number of semispaces at a given point exist, the max-min
convexity can be regarded as a multiorder
convexity~\cite{NS-08I,NS-08II}.

\begin{lemma}[Multiorder principle] \label{basic-tool} Let $X\subseteq \mmset^d$ and $p\in \mmset^d$. Then the following statements are equivalent:
\begin{enumerate}
\item[(i)] $p\in \conv(X)$;
\item[(ii)] for all $i\in I(p),$ there exists $x^i\in X$ such that $x^i\in \complement S_i(p)$.
\end{enumerate}
\end{lemma}

\begin{proof} (i) $\rightarrow$ (ii) By contradiction. Assume there is $i_0\in I(p)$ such that $X\cap \complement S_{i_0}(p)=\emptyset$. Then $p\in\conv(X)\subseteq S_{i_0}(p)$, in contradiction to $p\not \in  S_{i_0}(p)$.

(ii) $\rightarrow$ (i) By contradiction. Assume that $p\not \in
\conv(X)$. As $\conv(X)$ is a convex set, it follows from Theorem
\ref{p:semisp-conv} that there exists $i_0\in I(p)$ such that
$\conv(X)\subseteq S_{i_0}(p)$, which implies $\complement
S_{i_0}(p)\subseteq \complement \conv(X)$. But from (ii), there
exists $x_{i^0}\in \complement S_{i_0}(p)\cap \conv(X)$, which gives
a contradiction.
\end{proof}


\section{Separation and non-separation}\label{s:separation}

In what follows $\mmset^d$ has the usual Euclidean topology. If $A\subseteq \mmset^d$, we denote by $\overline{A}$ the closure of $A$, by $\text{int}(A)$ the interior of $A$ and by $\complement A$ the complement of $A$.

In the tropical convexity, all semispaces are open tropical halfspaces expressed
as solution sets to a strict two-sided max-linear inequality. See e.g. \cite{NS-1}. Thus the closures
of semispaces are hyperplanes.

In the case of max-min convexity, hyperplane in $\mmset^d$ can be defined as the solution set to a max-min linear equation
{\small
\begin{equation}\label{e:maxmin-hyper}
 \max(\min (a_1,x_1),\ldots,\min(a_d,x_d),a_{d+1})=
 \max(\min (b_1,x_1),\ldots,\min(b_d,x_d),b_{d+1}).
\end{equation}
}

The structure of a max-min hyperplane is presented in \cite{Nit-09}. One investigates the distribution of values for the left and right hand side of \eqref{e:maxmin-hyper}, and then identifies the regions in $\mmset^d$ where the values of the sides coincide. We illustrate this procedure in Figure 3, which shows the structure of a max-min hyperplane (line) in $\mmset^2$. The left side pictures show the distribution of values for both sides of \eqref{e:maxmin-hyper}: for the white regions the distribution is uniform and the value is equal to the coordinate of the finite point on the main diagonal that belongs to their boundary; the regions labeled $x_1$ are tiled by vertical lines for which the value of each point is equal to its $x_1$ coordinate, and the regions labeled $x_2$ are tiled by horizontal lines for which the value of each point is equal to its $x_2$ coordinate. The right side picture shows the graph of the line.

\begin{figure}[h]
    \centering
        \includegraphics[scale=0.7]{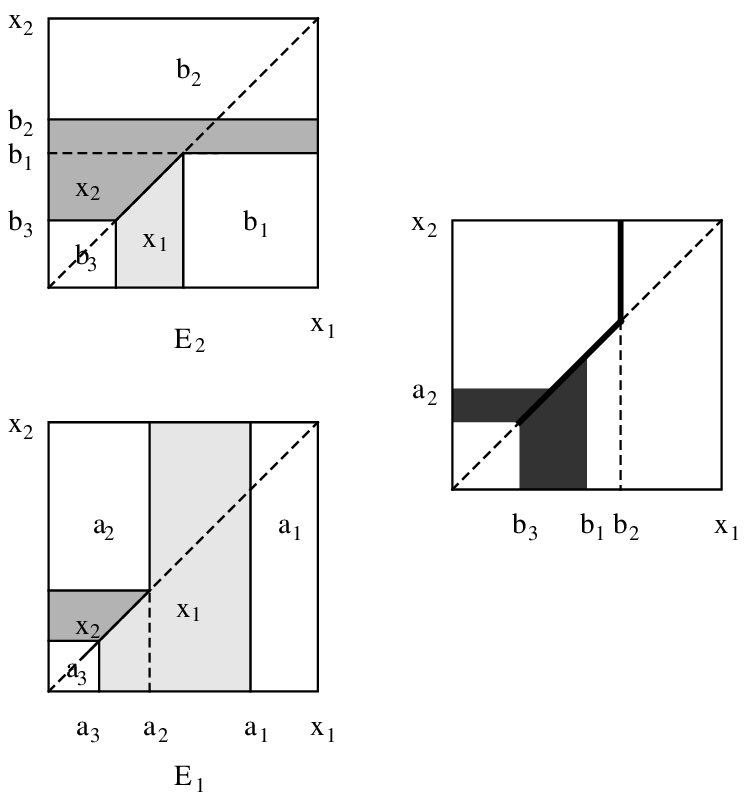}
        \label{figure1}
        \caption{A max-min hyperplane (line) in $\mmset^2$.}
    \end{figure}

In~\cite{N-Ser1} we investigated the relation between the max-min hyperplanes
and the closures of semispaces ${S}_i(x)$. We recall that the \emph{diagonal} of $\mmset^d$ is the set $\cD_d=\{(a,\ldots,a)\in \mmset^d\mid a\in\mmset\}.$

\begin{theorem}[\cite{N-Ser1}, Theorem 3.1]
\label{t:NitSer1}
A closure of semispace is a hyperplane if and only if it can be represented
as $\overline{S}_i(y)$ for some $y$ belonging to the diagonal.
\end{theorem}

Recall that a set $C\subset \mmset^d$ is separated from a point $x\in \mmset^d$ by a hyperplane $H\subset \mmset^d$ if
$C\subset H$ and $x\not\in H$.
Theorem \ref{t:NitSer1} shows exactly when classical separation by hyperplanes is possible.

\begin{corollary}[\cite{N-Ser1}, Corollary 3.3 and 3.4]
Let $x\in\mmset^d$, then
any closed max-min convex set $C\subseteq\cB^d$ not containing $x$
can be separated from $x$ by a hyperplane if and only if
$x$ lies on the diagonal.
\end{corollary}



In~\cite{N-Ser2}, we found a way to enhance separation by semispaces
showing that a point can be replaced by a box, i.e., a Cartesian product of
closed intervals. Namely, we investigated the separation of a box
$B=[\podx_1,\nadx_1]\times\ldots\times[\podx_d,\nadx_d]\subseteq \mmset^d$
from a max-min convex set $C\subseteq\mmset^d$, by which
we mean that there exists a set $S$ described in
Definition \ref{def:semi},
which contains $C$ and avoids $B$.

Assume that $\nadx_1\geq\ldots\geq\nadx_d$ and suppose that
$t(B)$ is the greatest integer such that
$\nadx_{t(B)}\geq \podx_i$ for all $1\leq i\leq t(B)$.
We will need the following condition:
\begin{equation}
\label{sep-cond}
\begin{split}
&\text{ If }(\nadx_1=1 )\ \&\ (y_l\geq\podx_l,1\le l\le d)\ \&\\
&(\nadx_l<y_l\ \text{for some $l\leq t(B)$}),\ \text{ then } y\notin C.
\end{split}
\end{equation}
Note that if the box is reduced to a point and if $\nadx_1=1$, then
$\nadx_l=1$ for all $l\leq t(B)$ so that $\nadx_l<y_l$ is
impossible. So \eqref{sep-cond} always
holds in the case of a point.

\if{
The formulation of our main result will also use an {\em
oracle} answering the question, whether or not
a given max-min convex set $C\subseteq\mmset^n$ lies in a semispace
$S$.
As in the conventional convex geometry
or tropical convex geometry, this question
can be answered in $O(mn)$ time if $C$ is a convex hull of
$m$ points. Indeed it suffices to answer whether any of the
inequalities defining
$S$ is satisfied for each of the $m$ points generating $C$.
}\fi

\begin{theorem}[\cite{N-Ser2}, Theorem~1]
\label{interval-sep}
Let $B=[\podx_1,\nadx_1]\times\ldots\times[\podx_d,\nadx_d]\subseteq \mmset^d$,
and let $C\subseteq\mmset^d$ be a max-min convex set avoiding
$B$. Suppose that $B$ and $C$ satisfy \eqref{sep-cond}.
Then there is a semispace
that contains $C$ and avoids $B$.
\end{theorem}

The box $B$ can be a point and in this case condition \eqref{sep-cond}
always holds. Therefore, some results
on max-min semispaces
\cite{NS-08II} can be deduced from Theorem \ref{interval-sep}.
The following is an immediate corollary
of Theorem \ref{interval-sep} and Proposition
\ref{p:semisp-conv}.

\begin{corollary}[\cite{NS-08II}]
\label{ns-08ii}
Let $x\in\mmset^d$ be non-increasing and $C\subseteq\mmset^d$ be a max-min
convex set avoiding $x$. Then $C$ is contained in one $S_i(x), i\in I(p),$
as in Definition \ref{def:semi}. Consequently these
sets are indeed the family of semispaces at $x$.
\end{corollary}

However, separation by semispaces is
impossible when $B, C$ do not satisfy
\eqref{sep-cond}.

\begin{theorem}[\cite{N-Ser2}, Theorem~2]
\label{BC:nonsep}
Suppose that
$B=[\podx_1,\nadx_1]\times\ldots\times[\podx_d,\nadx_d]\subseteq \mmset^d$
and the max-min convex set $C\subseteq\mmset^d$ are such that
$B\cap C=\emptyset$ but the condition \eqref{sep-cond} does not
hold. Then there is no semispace that contains $C$ and
avoids $B$.
\end{theorem}

In~\cite{N-Ser2} we also investigate the separation of max-min convex sets by a box,
and by a box and a semispace. We show that both kinds of separation are always possible if $n=2$,
but they are not valid in higher dimensions.

\section{Carath\'eodory theorems}\label{s:carath}

In this section we investigate classical convexity results in max-min setup.

\begin{theorem}[Carath\'eodory's theorem] Consider $X=\{x^1,x^2,\dots,x^m\}\subseteq \mmset^d,$ $m\ge d+1.$ Assume that $p\in \conv(X)$. Then there exists $X'=\{x'^i\vert i\in I\}\subseteq X,$ $1\le \vert I \vert \le d+1,$ such that $p\in \conv(X').$
\end{theorem}

\begin{proof} By Lemma \ref{basic-tool}, implication (i) $\rightarrow$ (ii), $p\in \conv(X)$ shows that for any $i\in I(p)$ there exists $x'^i\in X\cap \complement S_i(p)$. Define $X'=\{x'^i\vert i\in I(p)\}\subseteq X$. Then again by Lemma \ref{basic-tool}, now implication (ii) $\rightarrow$ (i), it follows that $p\in \conv(X').$
\end{proof}

\begin{theorem}[Colorful Carath\'eodory's theorem-weak form]\label{cara-color-weak} Let $X^0, X^1, \dots,$ $X^{d}$ be subsets in $\mmset^d$ and $p\in \mmset^d$. Assume that $p\in \conv(X^i)$ for all $0\le i\le d$. Then, up to a permutation of indices, there exist $x^i\in X^i,i\in I(p),$ such that $p\in \conv(\{x^i\vert i\in I(p)\}).$
\end{theorem}

\begin{proof} From Lemma \ref{basic-tool}, implication (i) $\rightarrow$ (ii), it follows that there exist $x^{i,j}\in X^i, 1\le i \le d+1, j\in I(p),$ such that $x^{i,j}\in \complement S_j(p),j\in I(p)$. Then again from Lemma~\ref{basic-tool}, implication (ii) $\rightarrow$ (i), and from $x^i:=x^{i,i}\in  \complement S_i(p), i\in I(p),$ it follows that $p\in \conv(\{x^i\vert i\in I(p)\}).$
\end{proof}

\begin{lemma}\label{l:sepincl} Let $p,q\in \mmset^d$. Then for all $i\in I(q)$ there exists $j\in I(p)$ such that $\complement S_j(p)\subseteq \complement S_i(q)$.
\end{lemma}

\begin{proof} The statement is equivalent to $S_i(q)\subseteq S_j(p)$. This follows from the fact that the convex set $S_i(q)$ has to be included in a semispace at $p$.
\end{proof}

We now explain the concept of internal separation property, in the
max-min setting. The proof of internal separation property is
deferred to the next section.

\begin{definition}
Given $X=\{x^0,\ldots,x^{d}\}\subseteq\mmset^d$, we say that a
finite point $p\in \conv(X)$ internally separates $x^0,\ldots,x^d$
if up to a permutation, each semispace $S_i(p), 0\le i \le d,$
corresponds to $x^i\in\complement S_i(p)$.
\end{definition}

\begin{theorem}
\label{t:intsep} For any subset
$X=\{x^0,\ldots,x^{d}\}\subseteq\mmset^d$, consisting of finite
points, $\conv(X)$ contains a point $p$ with internal separation
property.
\end{theorem}

We will need yet another simple observation, to obtain the colorful
Carath\'eodory theorem in most general form. Let $\mmseto$ be a
closed interval on the real line strictly containing $\mmset=[0,1]$, and
denote by $\underline{0}$, resp. $\overline{1}$ the least, resp. the
greatest element of $\mmseto$. We have
$\underline{0}<0<1<\overline{1}$, and we can define the max-min
semiring over $\mmseto$ with zero $\underline{0}$ and unity
$\overline{1}$. For $X\subseteq\mmset^d$, denote by $\convo(X)$ the
max-min convex hull of $X$ in $\mmseto^d$.

\begin{lemma}
\label{l:extend} For any $X\subseteq\mmset^d$, we have
$\convo(X)=\conv(X)$.
\end{lemma}
\begin{proof}
The ``new'' convex hull $\convo(X)$ is the set of combinations
\begin{equation}
\label{convoX}
\bigoplus_{i=1}^m \lambda_i\otimes x^i\colon m\geq 1,\ \lambda_i\in\mmseto,\ \bigoplus_{i=1}^m \lambda_i=\overline{1},
\end{equation}
taken for all $m$-tuples of points $x^i$ from $X$.

To obtain $\conv(X)\subseteq\convo(X)$, observe that when
$\lambda_i=1$ in~\eqref{convX} is changed to
$\lambda_i=\overline{1}$ the ``product'' $\lambda_i\otimes x^i$ is
unaffected (since all components of $x^i$ are $\leq 1$). To show
$\convo(X)\subseteq\conv(X)$, use the same observation to change
$\lambda_i=\overline{1}$ to $\lambda_i=1$ in~\eqref{convoX}. Next,
no combination~\eqref{convoX} (now with $1$ instead of
$\overline{1}$) has any negative components since all $x^i$ are
nonnegative and there is a point with coefficient $\lambda_i=1$.
Hence all $\lambda_i\colon \underline{0}\leq\lambda_i<0$ can be
changed to $0$ without affecting~\eqref{convoX}. This completes the
proof.
\end{proof}

\begin{corollary}
\label{c:extend}
A max-min convex set $C\subseteq\mmset^d$ remains
max-min convex in $\mmseto^d$.
\end{corollary}

\begin{theorem}[Colorful Carath\'eodory's theorem]\label{t:cara-color-strong}
Let $X^0, X^1, \dots,$ $X^{d}\subseteq\mmset^d$, and $C\subseteq
\mmset^d$ be a max-min convex set. Assume that $C\cap \conv(X^i)\not
=\emptyset$ for all $0\le i\le d$.
Then there exist $x^i\in X^i, 0\le i\le d,$ such that $C\cap \conv(\{x^0,x^1,\dots,x^{d}\})\not =\emptyset.$
\end{theorem}

\begin{proof} Assume first that all points in $X^0,X^1,\ldots, X^d$ are finite.
Take $p^i\in C\cap \conv(X^i), 0\le i\le d.$ By
Theorem~\ref{t:intsep} we can select a point $q$ which separates
$p^0,p^1,\dots,p^d$ internally, thus $p^i\in\complement S_i(q)$ for
all $i$. As $p^i\in C, 0\le i\le d,$ by Lemma \ref{basic-tool} one
has also $q\in C$. It remains to show that $q\in \conv(\{x^0,x^1,
\dots, x^{d}\})$, with some $x^i\in X^i, 0\le i\le d.$

By Lemma \ref{l:sepincl}, for any $0\le i\le d$, there exists $0\le
j\le d$ such that $\complement S_j(p^i)\subseteq \complement
S_i(q)$. As $p^i\in \conv(X^i)$, by Lemma \ref{basic-tool} there
exists $x^i\in X_i\cap \complement S_j(p_i)$. Hence $x^i\in
\complement S_i(q)$. Hence again by Lemma \ref{basic-tool} one has
$q\in \conv(\{x^0,x^1, \dots, x^{d}\})$. This proves the claim under
assumption that $X^0,X^1,\ldots, X^d$ have only finite points.

Without that assumption, regard $X^0, X^1, \dots,$
$X^{d},C\in\mmset^d$ as subsets of $\mmseto^d$ where $\mmseto$ is a
closed interval strictly containing $\mmset$. By
Corollary~\ref{c:extend}, $C$ remains max-min convex in $\mmseto^d$,
and by Lemma~\ref{l:extend} none of the convex hulls in the claim
change when they are considered in $\mmseto^d$. This extension makes
all points in $X^0,X^1,\ldots, X^d$ finite, and the previous
argument works in $\mmseto^d$ (with sectors in $\mmseto^d$).
\end{proof}


We conclude the section with the proof of internal separation
property in the cases when 1) $\conv(X)$ has a non-empty interior,
2) all vectors $p^{\ell}$ are non-increasing. These proofs can be
skipped by the reader, who can proceed to a general proof of
Theorem~\ref{t:intsep} written in the next section.

Let us introduce the notion of interior of a max-min convex set.

\begin{definition}
Interior of a max-min convex set $C\in\mmset^d$, denoted by
$\inter(C)$ is the subset of $C$ consisting of points $y$ such that
there is an open $d$-dimensional box
$(y_1-\epsilon,y_1+\epsilon)\times\cdots\times(y_d-\epsilon,y_d+\epsilon)\subseteq C$
for some $\epsilon>0$.
\end{definition}

\begin{proposition}\label{p:inter} Assume $X=\{x^0,x^1,\dots, x^{d}\}\subseteq \mmset^d$ generates a max-min
polytope $S=\conv(X)$ with non-empty interior. Then for any point
$p\in \inter(S)$ with all coordinates different, up to a permutation
of indices, one has $x^i\in \complement S_i(p), i\in I(p)$.
\end{proposition}

\begin{proof} We proceed by contradiction. As $p$ has all coordinates different and it is away from the boundary, the interiors of $\complement S_i(p), 0\le i\le d,$ are disjoint.
If $p$ does not internally separate the points of $X$, then there
exists $i\colon 0\le i\le d$ such that $\inter(\complement
S_i(p))\cap X=\emptyset$. However, as the complement
$\complement(\inter(\complement S_i(p)))$ is the topological closure
of $S_i(p)$, it is a max-min convex set, and hence $\conv(X)\cap
\inter(\complement S_i(p))=\emptyset$. But then $p$ is not in the
interior of $\conv(X)$.
\end{proof}

The notion of interior and, more generally, of dimension in max-min
convexity will be investigated in another publication. We now treat
the other special case.

\begin{proposition}\label{propo-same-order}  Assume that $x^{\ell}\in \mmset^d, 0\le \ell\le d,$ are non-increasing, i.e.,
\begin{equation}\label{eq-ineeq-34}
x^{\ell}_1\geq x^{\ell}_2\geq\ldots\geq x^{\ell}_d,\quad 0\le {\ell}\le d,
\end{equation}
and finite. Then there exists $p\in\mmset^d$ such that
$x^\ell\in\complement S_{\ell}(p)$ for all $\ell\in\{0,1,\ldots,
d\}$.
\end{proposition}
\begin{proof}
Let $y_d:=\max_{\ell=0}^d x_d^{\ell}$, and $\ell'_1$ be an index
where this maximum is attained. Reordering the points, we can assume
$\ell'_1=d$. Let $y_{d-1}:=\max_{\ell=0}^{d-1} x_{d-1}^{\ell}$ and
$\ell'_2$ be an index where this maximum is attained. Reordering the
points $x^0,\ldots, x^{d-1}$ we can assume $\ell'_2=d-1$. On a
general step of this procedure, we have obtained the partial maxima
$y_d, y_{d-1},\ldots, y_{d-t+1}$ equal to $x_d^{d}$,
$x_{d-1}^{d-1},\ldots, x_{d-t+1}^{d-t+1}$ (having reorganized the
given points $x$), and we define $y_{d-t}:=\max_{\ell=0}^{d-t}
x_{d-t}^{\ell}$, requiring that $y_{d-t}=x_{d-t}^{d-t}$. On the last
step, we have $y_1=\max(x_1^0,x_1^1)$ and swap $x^0$ with $x^1$ (if
necessary) to obtain $y_1=x_1^1$.

This process defines the vector $y=(y_1,\ldots,y_d)$ and rearranges
the given points $x^{0},\ldots,x^{d}$ in such a way that
\begin{equation}
y_t={\max}_{\ell=0}^t x_t^\ell=x_t^t,\quad\forall t\in\{1,\ldots,d\}.
\end{equation}

Now define $p$ to be the largest non-increasing vector satisfying
$p\leq y$. We will show that $p$ is a point that we need. Before the
main argument we observe that
\begin{equation}
\label{e:zprop1}
p_t\leq y_t=x_t^t\quad \forall t\in\{1,\ldots,d\},
\end{equation}
and
\begin{equation}
\label{e:zprop2}
\left\{
\begin{split}
p_1&=y_1,\\
p_t&=y_t={\max}_{\ell=0}^t x_t^\ell=x_t^t\quad \text{if $p_t<p_{t-1}$}.
\end{split}
\right .
\end{equation}
Only~\eqref{e:zprop2} has to be shown. Indeed, if $p_1<y_1$, then
$(y_1,p_2,\ldots,p_d)$ is a non-increasing vector bounded by $y$
from above and contradicting the maximality of $z$, so $p_1=y_1$
holds. If $p_t<p_{t-1}$ and $p_t<y_t$ then defining
$p'_t:=\min(p_{t-1},y_t)$ we have $p_{t-1}\geq p'_t\geq p_{t+1}$ and
$p'_t\leq y_t$, so again,
$(p_1,\ldots,p_{t-1},p'_t,p_{t+1},\ldots,p_d)$ is a non-increasing
vector bounded by $y$ from above and contradicting the maximality of
$p$.

For what follows, we refer the reader to Definition \ref{def:semi}, that describes the structure of the semispaces.

We now show that $x^\ell\in\complement S_\ell(p)$ for all
$\ell\in\{0,1,\ldots,d\}$, starting with $\ell=0$. In this case we
need to argue that $x_t^0\leq p_t$ for all $t$. Indeed, when
$p_{t-1}>p_t$, the inequality $x_t^0\leq p_t$ follows
from~\eqref{e:zprop2} (second part). If $p_{t-1}=p_t$, then either
$p_1=\ldots=p_t$, or $p_{t-i-1}>p_{t-i}=\ldots=p_{t-1}=p_t$. In the
first case we have~$x_s^0\leq p_s$ for $s=1$, and in the second case
for $s=t-i$, and in both cases the required inequality $x_t^0\leq
p_t$ follows since $x^0$ is a non-increasing vector.

When $\ell>0$ and $p_{\ell-1}>p_\ell$, we have $x_\ell^\ell=p_\ell$
by~\eqref{e:zprop2}, so $x_\ell^\ell\geq p_\ell$. When
$p_{t-1}>p_t$, the inequalities $x_t^\ell\leq p_t$ for $t>\ell$
follow from~\eqref{e:zprop2}, and when $p_{t-1}=p_t$, we have
$p_{t-i-1}>p_{t-i}=\ldots=p_t$ for some $i$, where $t-i\geq \ell$.
In this case $x_{t-i}^{\ell}\leq p_{t-i}$ follows
from~\eqref{e:zprop2}, and we use that $x^\ell$ is non-increasing to
obtain $x_t^\ell\leq p_t=p_{t-i}$.

If $p_{\ell-1}=p_\ell$, then either $p_\ell=p_{\ell+1}=\ldots=p_d$,
or there exists $i$ such that
$p_\ell=\ldots=p_{\ell+i}>p_{\ell+i+1}$. In this case
$x_\ell^\ell\geq p_\ell$ follows from~\eqref{e:zprop1}, and the
inequalities $x_t^\ell\leq p_t$ for $t>\ell+i$ are shown as in the
previous case.

The proof is complete.
\end{proof}

\if{
\begin{proposition}\label{propo-dim2}
Let $(x_1,y_1),\,(x_2,y_2),\,(x_3,y_3)\in\mmset^2$, with neither of the
points at the boundary. Then there exists a point internally separating them.
\end{proposition}

\begin{proof}
Denote $t:=\min_{i=1}^3 \max(x_i,y_i)$, and let this minimum be attained by
$(x_1,y_1)$, with $x_1\leq y_1=t$.
Consider the following subsets (boxes) of $\mmset^2$:
\begin{itemize}
\item $B_0=\{(x,y)\in \mmset^2\vert x\le y_1, y\le y_1\}$,
\item $B_1=\{(x,y)\in \mmset^2\vert x\le y_1, y\ge y_1\}$,
\item $B_2=\{(x,y)\in \mmset^2\vert x\ge y_1, y\ge y_1\}$,
\item $B_3=\{(x,y)\in \mmset^2\vert x\ge y_1, y\le y_1\}$.
\end{itemize}

Box $B_0$ contains only the point $(x_1,y_1)$. The other points belong to the other boxes. One identifies the following cases:
\begin{itemize}
\item[1.] $(x_2,y_2), (x_3,y_3)\in B_2$,
\item[2.] $(x_2,y_2)\in B_2$, $(x_3,y_3)\in B_1$,
\item[3.] $(x_2,y_2)\in B_2$, $(x_3,y_3)\in B_3$,
\item[4.] $(x_2,y_2)\in B_1$, $(x_3,y_3)\in B_3$,
\item[5.] $(x_2,y_2), (x_3,y_3)\in B_1$,
\item[6.] $(x_2,y_2), (x_3,y_3)\in B_3$.
\end{itemize}

In cases 1. through 4., one can choose $(y_1,y_1)$ as a separating point.

In case 5., we see that all points satisfy $x_i\leq y_i$ for $i=1,2$ and $3$, so
we apply Proposition~\ref{propo-same-order}.

In case 6., assume without loss of generality that $y_3\leq y_2\leq y_1$. If
$x_3\leq x_2$, then we can choose $(x_3,y_2)$ as a separating point, and
if $x_3\geq x_2$ then we choose $(x_2,y_2)$.
\end{proof}
}\fi

\section{Internal separation property}
\label{s:intsep} This section is devoted to the proof of
Theorem~\ref{t:intsep} (the internal separation property). Let
$u^{(i)}$, for $i=1,\ldots,d+1$ be the given points in $\mmset^d$, let $h\in\mmset$
and let $A\in\mmset^{(d+1)\times d}$ be the matrix where these
vectors are rows. For such a matrix, denote by $A^{(h)}$ the Boolean
matrix with entries
\begin{equation}
a_{ij}^{(h)}=
\begin{cases}
1, &\text{if $a_{ij}\geq h$},\\
0, &\text{if $a_{ij}<h$}.
\end{cases}
\end{equation}
Following the literature on max-min algebra, we may call it the {\em
threshold matrix} of level $h$. Let $t$ be the greatest $h$ for
which $A^{(h)}$ contains a $d\times d$ submatrix with a nonzero
permanent (in other words, a permutation with nonzero weight).

For every $h>t$, every $d\times d$ submatrix of $A^{(h)}$ has zero
permanent. Take $h>t$ to be smaller than any entry of $A$ that is
greater than $t$, and consider the bipartite graph corresponding to
$A^{(h)}$\footnote{One part of the vertices represents the rows, and
the other represents the columns. The graph contains an edge between
the row vertex $i$ and the column vertex $j$ if and only if
$a_{ij}^{(h)}=1$, that is, $a_{ij}\geq h$.}. As $A^{(h)}$ has zero
permanent, the size of maximal matching in that graph is less than
$d$. By the K\"{o}nig theorem, the size of maximal matching is equal
to the size of the minimal vertex cover. In particular, there exists
a subset of rows $M_2$ and a subset of columns $N_2$ with number of
elements $m_2$ and $n_2$ respectively, such that $m_2+n_2<d$ and
such that all $1$'s of $A^{(h)}$ are in these columns and rows. Let
$M_1$, resp. $N_1$, be the complements of $M_2$, resp. $N_2$ in
$\{1,\ldots,d+1\}$, resp. $\{1,\ldots,d\}$. Then all entries of the
submatrix $A_{M_1N_1}^{(h)}$ are zero, and hence all entries of
$A_{M_1N_1}$ are less than or equal to $t$, and we have
$m_1+n_1>d+1$, where $m_1$, resp. $n_1$ are the number of elements
in $M_1$, resp. $N_1$.

Thus $A$ contains an $m_1\times n_1$ submatrix $B^{\leq
t}:=A_{M_1N_1}$ where all entries do not exceed $t$ and we have
$m_1+n_1>d+1$. At the same time, there is a row index $f$ which we
call the {\em free index}, and a permutation $\pi\colon
\{1,\ldots,d+1\}\bez\{f\}\mapsto\{1,\ldots,d\}$ such that
$a_{i\pi(i)}\geq t$ for all $i\neq f$. The pair $(B^{\leq t},\pi)$
will be called a {\em (K\"{o}nig) diagram}. Denote the number of
intersections of $\pi$ with $A_{M_1N_1}$ by $r$ and with
$A_{M_2N_2}$ by $s$. Then we obtain, having $d$ as the sum of the
number of intersections of $\pi$ with $A_{M_1N_1}$, $A_{M_1N_2}$,
$A_{M_2N_1}$ and $A_{M_2N_2}$ that
\begin{equation}
\label{nequalsto}
d=
\begin{cases}
r+(m_1-r-1)+(n_1-r)+s, &\text{if $f\in M_1$},\\
r+(m_1-r)+(n_1-r)+s, &\text{if $f\notin M_1$}.
\end{cases}
\end{equation}
Eliminating $r$ from~\eqref{nequalsto} we obtain
\begin{equation}
\label{lequalto}
r=
\begin{cases}
m_1+n_1-(d+1)+s,&\text{if $f\in M_1$},\\
m_1+n_1-d+s, &\text{if $f\notin M_1$}.
\end{cases}
\end{equation}

We see that with $m_1$, $n_1$ and $d$ fixed, the number $r$ is
minimal when $f\in M_1$ and $s=0$. Such diagrams will be called {\em
tight}. See Figure 4 for an illustration of a tight diagram. The entries in $\pi$ are represented by *.
In general, the {\em tightness} of a diagram is defined as
the non-positive integer $m_1+n_1-d-1-r$.

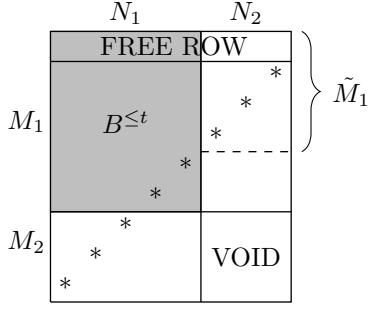
\begin{figure}[h]
\begin{tikzpicture}[scale=.4]
\draw [fill=lightgray] (0,3)--(5,3)--(5,9)--(0,9)--(0,3);

\draw [line width = 0.5] (0,0)--(8,0)--(8,9)--(0,9)--(0,0); \draw
[line width = 0.5] (0,8)--(8,8); \draw [line width = 0.5]
(0,3)--(8,3); \draw [line width = 0.5] (5,0)--(5,9);

\node at (2.5,6) {$B^{\leq t}$}; \node at (6.5,1.5) {VOID}; \node at
(4.1,8.5) {FREE ROW}; \node at (-.8,6) {$M_1$}; \node at (-.8,2)
{$M_2$}; \node at (2.5,9.6) {$N_1$}; \node at (6.5,9.6) {$N_2$};

\node at (.5,.5) {*};
\node at (1.5,1.5) {*};
\node at (2.5,2.5) {*};
\node at (3.5,3.5) {*};
\node at (4.5,4.5) {*};

\draw [dashed, line width=0.5] (5,5) -- (8,5);


\draw [decorate,decoration={brace,amplitude=7pt,mirror,raise=4pt},yshift=0pt]
(8,5) -- (8,9) node [black,midway,xshift=0.8cm] {$\Tilde{M}_1$};

\node at (5.5,5.5) {*};
\node at (6.5,6.5) {*};
\node at (7.5,7.5)
{*};
\end{tikzpicture}
\caption{A tight diagram ($\Tilde{M}_1$ is the set appearing in the
proof of Theorem~\ref{t:intsep} in the end of this section).
\label{f:tight}}
\end{figure}

Let us indicate some sufficient conditions for $(B^{\leq t},\pi)$ to
be tight (the proof is omitted).

\begin{lemma}
\label{l:simpletight} The diagram $(B^{\leq t},\pi)$ is tight if $m_1+n_1=d+2$, $f\in M_1$ and $\pi$ intersects
with $B^{\leq t}$ only once. In particular, if $B^{\leq t}$ is a column, then $(B^{\leq t},\pi)$ is tight.
\end{lemma}

\begin{proof} Substituting $m_1+n_1=d+2$ and $r=1$ in the first line
of~\eqref{lequalto} we have $s=0$.
\end{proof}

Our next aim is to show that there always exists at least one tight
diagram, and let us start with a pair of auxiliary lemmas.

\begin{lemma}[Sinking]
\label{l:sinking} Let $(B^{\leq t},\pi)$ be not tight, and let
$(k_0,\pi(k_0))\in M_1\times
 N_1$. Then we have one of the following alternatives:
\begin{itemize}
\item[{\rm(i)}] There exists a sequence $k_0,\ldots, k_l$ such that
$(k_i,\pi(k_i))\in M_2\times N_1$ for $i=1,\ldots,l-1$,
$(k_l,\pi(k_l))\in M_2\times N_2$ or $k_l$ is free, and
$a_{k_i\pi(k_{i-1})}>t$ for all $i=1,\ldots,l$;
\item[{\rm(ii)}] There is a tight diagram $(\Tilde{B}^{\leq t},\pi)$.
\end{itemize}
\end{lemma}
\begin{proof}[Proof (see Figures~\ref{f:sinking}
and~\ref{f:tightsinking})] If we have $a_{i\pi(k_0)}\leq t$ for all
$i$, then the entire column with index $\pi(k_0)$ can be taken for
$\Tilde{B}^{\leq t}$, that is $M_1=\{1,\dots,d+1\}, N_1=\pi(k_0)$
and the diagram $(\Tilde{B}^{\leq t},\pi)$ is tight (by
Lemma~\ref{l:simpletight}). If this is not the case, select $k'_1\in
M_2$ with $a_{k'_1\pi(k_0)}>t$.
Then we proceed as in the following general description (with the
sequence $k_0, k'_1$).

In general, suppose that we have found a sequence of rows $k_0,
k'_1,\ldots, k'_l$ where $k_0\in M_1$, $k'_1,\ldots, k'_l\in M_2$
and $\pi(k_0),\pi(k'_1),\ldots,\pi(k'_{l-1})\in N_1$ with the
following property:

(*) For each $s\colon 1\leq s\leq l$ there is a subsequence
$k_0,k_1\ldots k_r$ of $k_0,k'_1,\ldots, k'_s$ such that $k_r=k'_s$
and $a_{k_i\pi(k_{i-1})}>t$ for all $i=1,\ldots,r$.

If $\pi(k'_l)$ is in $N_2$ or $k'_l$ is free then we are done.
Otherwise consider the submatrix extracted from the columns
$\pi(k_0),\pi(k'_1),\ldots,\pi(k'_l)$ and all rows except for
$k'_1,\ldots, k'_l$. If this submatrix does not contain any entries
greater than $t$ then it can be taken for $\Tilde{B}^{\leq t}$ and
the diagram $(\Tilde{B}^{\leq t},\pi)$ is tight by
Lemma~\ref{l:simpletight}. Otherwise we choose
$k'_{l+1}\notin\{k_0,k'_1,\ldots,k'_l\}$ in $M_2$ in such a way that
$a_{k'_{l+1}\pi(i)}>t$ for some $i$ in $\{k_0,k'_1,\ldots,k'_l\}$.
Then $k_0,k'_1,\ldots, k'_l,k'_{l+1}$ satisfies the property (*),
and the process is continued until the intersection of $\pi$ with
$M_2\times N_1$ is exhausted and we end up either with a free $k_l$,
or such that $(k_l,\pi(k_l))\in M_2\times N_2$.
\end{proof}

\begin{figure}[h]
\begin{tikzpicture}[scale=.4]
\draw [fill=lightgray] (0,5)--(5,5)--(5,9)--(0,9)--(0,5);

\draw [line width = 0.5] (0,0)--(8,0)--(8,9)--(0,9)--(0,0); \draw
[line width = 0.5] (0,1)--(8,1); \draw [line width = 0.5]
(0,5)--(8,5); \draw [line width = 0.5] (5,0)--(5,9);

\draw [dotted,line width = 0.5,->] (5.5,1.7)--(5.5,5.3); \node at
(5.5,5.5){$\bullet$};

\draw [line width = 0.5,->] (4.5,2.4)--(4.5,1.7); \node at (4.5,1.5)
{$\bullet$};

\draw [line width = 0.5,->] (2.5,4.4)--(2.5,2.7); \node at (2.5,2.5)
{$\bullet$};

\draw [line width = 0.5,->] (1.5,5.4)--(1.5,4.7); \node at (1.5,4.5)
{$\bullet$};

\node at (2.5,7) {$B^{\leq t}$}; \node at (4.1,.5) {FREE ROW}; \node
at (-.8,7) {$M_1$}; \node at (-.8,3) {$M_2$}; \node at (2.5,9.6)
{$N_1$}; \node at (6.5,9.6) {$N_2$};

\node at (.5,6.5) {*};
\node at (1.5,5.5) {*};
\node at (2.5,4.5) {*};
\node at (3.5,3.5) {*};
\node at (4.5,2.5) {*};
\node at (5.5,1.5) {*};

\draw [fill=lightgray] (10,5)--(15,5)--(15,9)--(10,9)--(10,5);

\draw [line width = 0.5] (10,0)--(18,0)--(18,9)--(10,9)--(10,0);
\draw [line width = 0.5] (10,1)--(18,1); \draw [line width = 0.5]
(10,5)--(18,5); \draw [line width = 0.5] (15,0)--(15,9);

\draw [line width = 0.5,->] (10.5,6.4)--(10.5,3.7); \node at
(10.5,3.5) {$\bullet$};

\draw [line width = 0.5,->] (13.5,3.6)--(13.5,.7); \node at
(13.5,.5) {$\bullet$};

\node at (12.5,7) {$B^{\leq t}$}; \node at (13.7,.5) {FREE ROW};
\node at (9.4,7) {$M_1$}; \node at (9.4,3) {$M_2$}; \node at
(12.5,9.5) {$N_1$}; \node at (16.5,9.5) {$N_2$};

\node at (10.5,6.5) {*};
\node at (11.5,5.5) {*};
\node at (12.5,4.5) {*};
\node at (13.5,3.5) {*};
\node at (14.5,2.5) {*};
\node at (15.5,1.5) {*};
\end{tikzpicture}
\caption{Possible outcomes of sinking (the free row could belong to
$M_1$ but then the outcome on the right is
impossible).\label{f:sinking}}
\end{figure}
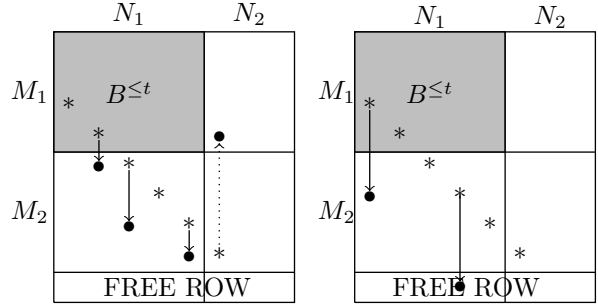

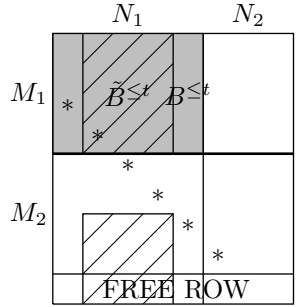
\begin{figure}[h]
\begin{tikzpicture}[scale=.4]
\draw [fill=lightgray] (0,5)--(5,5)--(5,9)--(0,9)--(0,5);

\draw [line width = 0.5] (0,0)--(8,0)--(8,9)--(0,9)--(0,0); \draw
[line width = 0.5] (0,1)--(8,1); \draw [line width = 1]
(0,5)--(8,5); \draw [line width = 0.5] (5,0)--(5,9);

\draw [line width = 0.5] (1,9)--(1,5)--(4,5)--(4,9);

\draw (1,5)--(4,8);
\draw (1,6)--(4,9);
\draw (1,7)--(3,9);
\draw (1,8)--(2,9);
\draw (2,5)--(4,7);
\draw (3,5)--(4,6);

\draw [line width = 0.5] (1,0)--(1,3)--(4,3)--(4,0); \draw
(1,0)--(4,3); \draw (2,0)--(4,2); \draw (3,0)--(4,1); \draw
(1,1)--(3,3); \draw (1,2)--(2,3);

\node at (4.5,7) {$B^{\leq t}$}; \node at (2.5,7) {${\tilde B}^{\leq
t}$}; \node at (4.1,.5) {FREE ROW}; \node at (-.8,7) {$M_1$}; \node
at (-.8,3) {$M_2$}; \node at (2.5,9.6) {$N_1$}; \node at (6.5,9.6)
{$N_2$};

\node at (.5,6.5) {*};
\node at (1.5,5.5) {*};
\node at (2.5,4.5) {*};
\node at (3.5,3.5) {*};
\node at (4.5,2.5) {*};
\node at (5.5,1.5) {*};
\end{tikzpicture}
\caption{A tight diagram arising when the sinking
stops.\label{f:tightsinking}}
\end{figure}

Now we consider a reverse process.

\begin{lemma}[Lifting]
\label{l:lifting} Let $(B^{\leq t},\pi)$ be not tight, and let
$(k_0,\pi(k_0))\in M_2\times N_2$. Then we have one of the following alternatives:
\begin{itemize}
\item[{\rm(i)}] There exists a sequence $k_0,\ldots, k_l$ such that
$(k_i,\pi(k_i))\in M_1\times N_2$ for $i=1,\ldots,l-1$,
$(k_l,\pi(k_l))\in M_2\times N_2$ or $k_l$ is free, and
$a_{k_i\pi(k_{i-1})}>t$ for all $i=1,\ldots,l$;
\item[{\rm(ii)}] There is a tighter diagram $(\Tilde{B}^{\leq t},\pi)$.
\end{itemize}
\end{lemma}
\begin{proof}[Proof (see Figures~\ref{f:lifting} and~\ref{f:tightlifting})]
If we have $a_{i\pi(k_0)}\leq t$ for all
$i$, then the column index $\pi(k_0)$ can be added to $M_1$ and the
resulting diagram $(\Tilde{B}^{\leq t},\pi)$ is tighter (i.e., has a
greater tightness) than $(B^{\leq t},\pi)$, since the size of
$\Tilde{B}^{\leq t}$ increased while the number of intersections
with $\pi$ is the same. Otherwise we can select $k'_1\in M_1$ with
$a_{k'_1\pi(k_0)}>t$
and proceed as in the following general description (with the
sequence $k_0, k'_1$).

In general, suppose that we have found a sequence of rows $k_0,
k'_1,\ldots, k'_l$ where $k_0\in M_2$, $k'_1,\ldots, k'_l\in M_1$
and $\pi(k_0),\pi(k'_1),\ldots,\pi(k'_l)\in N_2$ with the property
(*) in the proof of Lemma~\ref{l:sinking}.

\if{(*) For each $s\colon 1\leq s\leq l$ there is a subsequence
$k_0,k_1\ldots k_r$ of $k_0,k'_1,\ldots, k'_s$ such that $k_r=k'_s$
and $a_{k_i\pi(k_{i-1}}>t$ for all $i=1,\ldots,r$.}\fi

If $\pi(k'_l)$ is in $N_1$ or is free then we are done. Otherwise
consider the submatrix extracted from the columns of $N_1$ and
$\pi(k_0),\pi(k'_1),\ldots,\pi(k'_l)$, and all rows of $M_1$ except
for $k'_1,\ldots, k'_l$. If this submatrix does not contain any
entries greater than $t$ then it can be taken for $\Tilde{B}^{\leq
t}$ and the diagram $(\Tilde{B}^{\leq t},\pi)$ is tighter than
$(B^{\leq t},\pi)$ since the sum of dimensions increases by one but
the number of intersections of $\pi$ with $\Tilde{B}^{\leq t}$ is
the same. Otherwise we choose
$k'_{l+1}\notin\{k_0,k'_1,\ldots,k'_l\}$ in $M_2$ in such a way that
$a_{k'_{l+1}\pi(i)}>t$ for some $i$ in $\{k_0,k'_1,\ldots,k'_l\}$.
Then the sequence $k_0,k'_1,\ldots, k'_l,k'_{l+1}$ satisfies the
property (*), and the process is continued until the intersection of
$\pi$ with $M_1\times N_2$ is exhausted and we end up either with a
free $k_l$, or such that $(k_l,\pi(k_l))\in M_1\times N_1$.
\end{proof}

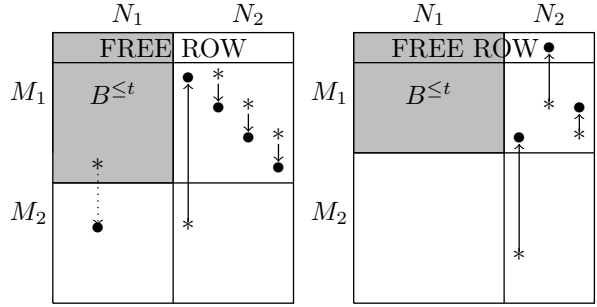
\begin{figure}[h]
\begin{tikzpicture}[scale=.4]
\draw [fill=lightgray] (0,4)--(4,4)--(4,9)--(0,9)--(0,5);

\draw [line width = 0.5] (0,0)--(8,0)--(8,9)--(0,9)--(0,0); \draw
[line width = 0.5] (0,8)--(8,8); \draw [line width = 0.5]
(0,4)--(8,4); \draw [line width = 0.5] (4,0)--(4,9);

\node at (2,7) {$B^{\leq t}$}; \node at (4,8.5) {FREE ROW}; \node at
(-.8,7) {$M_1$}; \node at (-.8,3) {$M_2$}; \node at (2.5,9.6)
{$N_1$}; \node at (6.5,9.6) {$N_2$};

\node at (1.5,4.5) {*};


\draw [dotted,line width = .5,->] (1.5, 4.3)--(1.5,2.6);
\node at (1.5,2.5) {$\bullet$};

\node at (5.5,7.5) {*};

\draw [line width = 0.5,->] (5.5, 7.3)--(5.5,6.7); \node at
(5.5,6.5) {$\bullet$};

\node at (6.5,6.5) {*};

\draw [line width = 0.5,->] (6.5,6.3)--(6.5,5.7); \node at (6.5,5.5)
{$\bullet$};

\node at (7.5,5.5) {*};

\draw [line width = 0.5,->] (7.5,5.3)--(7.5,4.7); \node at (7.5,4.5)
{$\bullet$};

\node at (4.5,2.5) {*};

\draw [line width = 0.5,->] (4.5, 2.7)--(4.5,7.3); \node at
(4.5,7.5) {$\bullet$};

\draw [fill=lightgray] (10,5)--(15,5)--(15,9)--(10,9)--(10,5);

\draw [line width = 0.5] (10,0)--(18,0)--(18,9)--(10,9)--(10,0);
\draw [line width = 0.5] (10,8)--(18,8); \draw [line width = 0.5]
(10,5)--(18,5); \draw [line width = 0.5] (15,0)--(15,9);

\node at (12.5,7) {$B^{\leq t}$}; \node at (13.7,8.5) {FREE ROW};
\node at (9.2,7) {$M_1$}; \node at (9.2,3) {$M_2$}; \node at
(12.5,9.6) {$N_1$}; \node at (16.5,9.6) {$N_2$};

\node at (16.5,6.5) {*};

\draw [line width = 0.5,->] (16.5,6.7)--(16.5,8.3); \node at
(16.5,8.5) {$\bullet$};

\node at (17.5,5.5) {*};

\draw [line width = 0.5,->] (17.5,5.7)--(17.5,6.3); \node at
(17.5,6.5) {$\bullet$};

\node at (15.5,1.5) {*};

\draw [line width = 0.5,->] (15.5,1.7)--(15.5,5.3); \node at
(15.5,5.5) {$\bullet$};

\end{tikzpicture}
\caption{Possible outcomes of lifting (the free row could also
belong to $M_2$).\label{f:lifting}}
\end{figure}

\begin{figure}[h]
\begin{tikzpicture}[scale=.4]
\draw [fill=lightgray] (0,4)--(4,4)--(4,9)--(0,9)--(0,4);

\draw [line width = 0.5] (0,0)--(8,0)--(8,9)--(0,9)--(0,0);
\draw [line width = 0.5] (0,4)--(8,4); \draw [line width = 0.5]
(4,0)--(4,9);

\draw [line width = 0.5] (0,6)--(5,6)--(5,9); \draw [line width =
0.5] (7,9)--(7,6)--(8,6);

\draw (0,7)--(2,9);
\draw (0,6)--(3,9);
\draw (1,6)--(4,9);
\draw (2,6)--(5,9);
\draw (3,6)--(5,8);
\draw (0,8)--(1,9);
\draw (7,8)--(8,9);
\draw (7,7)--(8,8);
\draw (7,6)--(8,7);
\draw (4,6)--(5,7);

\draw [line width = 0.5] (0,4)--(5,4)--(5,5)--(0,5); \draw [line
width = 0.5] (7,4)--(7,5)--(8,5); \draw (0,4)--(1,5); \draw
(1,4)--(2,5); \draw (2,4)--(3,5); \draw (3,4)--(4,5); \draw
(4,4)--(5,5); \draw (7,4)--(8,5);

\node at (3.2,5.5) {$B^{\leq t}$};
\node at (2.5,7.5) {${\tilde B}^{\leq t}$};
\node at (-.8,7) {$M_1$}; \node at (-.8,3) {$M_2$}; \node at
(2.5,9.6) {$N_1$}; \node at (6.5,9.6) {$N_2$};

\node at (6.5,6.5) {*};
\node at (7.5,5.5) {*};

\node at (4.5,2.5) {*};
\node at (5.5,7.5) {*};
\end{tikzpicture}
\caption{A tighter diagram arising when the lifting
stops.\label{f:tightlifting}}
\end{figure}
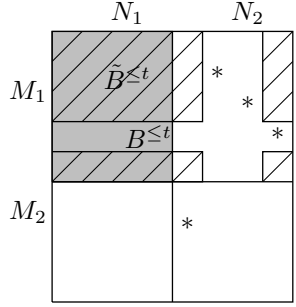

We mainly need to show the following.

\begin{lemma}[Diagram Improvement]
\label{l:improve}
 If $(B^{\leq t},\pi)$ is not tight, then there is a tighter
diagram $(\Tilde{B}^{\leq t},\Tilde{\pi})$.
\end{lemma}
\begin{proof}
By contradiction, suppose that a tighter diagram does not exist.
Then, Lemma~\ref{l:sinking} yields a sequence
$k_{l_0},k_{l_0+1}\ldots, k_{m_0}$, where $l_0=0$,
 $(k_{l_0},\pi(k_{l_0}))\in
M_1\times N_1$, $(k_{m_0},\pi(k_{m_0}))\in M_2\times N_2$ or
$k_{m_0}$ is free, $(k_s,\pi(k_s))\in M_2\times N_1$ for all
$s=l_0,\ldots,m_0-1$, and $a_{k_s\pi(k_{s-1})}>t$ for all
$s=l_0+1,\ldots,m_0$.

If $k_{m_0}$ is free then we define $\Tilde{\pi}$ by
$\Tilde{\pi}(k_s):=\pi(k_{s-1})$ for $s=l_0+1,\ldots,m_0$. The row
$k_{l_0}$ becomes free, and for all the remaining indices $i$ we
define $\Tilde{\pi}(i):=\pi(i)$. We see that the number of
intersections of $\Tilde{\pi}$ with $B^{\leq t}$ is one less than
that of $\pi$ with $B^{\leq t}$, hence $(B^{\leq t},\Tilde{\pi})$ is
tighter.

Otherwise, Lemma~\ref{l:lifting} yields a sequence
$k_{m_0},k_{m_0+1}\ldots, k_{l_1}$, where
 $(k_{m_0},\pi(k_{m_0}))\in
M_2\times N_2$, $(k_{l_1},\pi(k_{l_1})\in M_1\times N_1$ or
$k_{l_1}$ is free, $(k_s,\pi(k_s))\in M_1\times N_2$ for all
$s=m_0,\ldots,l_1-1$, and $a_{k_s\pi(k_{s-1})}>t$ for all
$s=m_0+1,\ldots,l_1$.

If $k_{l_1}$ is free, then the diagram can be made tighter as above,
replacing $m_0$ with $l_1$ in the definition of $\Tilde{\pi}$.

The composition of sinking and lifting, or if any of these
procedures end up with a free row index, will be called a {\em
(full) turn} of the {\em trajectory}.

The sinking and lifting procedures are then applied again and again,
until either one of the following holds.

a) On some turn, let it be turn number $(s+1)$, we encounter a row
index $k_{l_s+t}, t\ge 1,$ which is already in the trajectory, written as
$k_{l_r+t'}$ (with $r<s$ or $t'=0$ and $r\leq s$). In this case we
make a cyclic trajectory $k_{l_s},k_{l_s+1},\ldots, k_{l_s+t},
k_{l_r+t'+1},\ldots, k_{l_s}$
where no two intermediate indices are repeated.

b) There are no repetitions 
but we meet a free row index in the end.

In both cases, let $p$ be the length of the trajectory, and rename
the indices of the resulting cyclic trajectory without repetitions,
or the resulting acyclic trajectory ending with the free row index,
to $l_0,l_1,\ldots,l_p$. Clearly, for any two adjacent indices $l_s$
and $l_{s+1}$ of this trajectory, we have $a_{l_{s+1}\pi(l_s)}>t$,
and either $(l_{s+1},\pi(l_s))\in M_1\times N_2$, or
$(l_{s+1},\pi(l_s))\in M_2\times N_1$. This shows that defining
$\Tilde{\pi}$ by $\Tilde{\pi}(l_s)=\pi(l_{s-1})$ for $s=1,\ldots,p$,
setting $l_p$ as the new free row in case b), and defining
$\Tilde{\pi}(i):=\pi(i)$ for all the remaining row indices, we
obtain a tighter diagram $(B^{\leq t},\Tilde{\pi})$, since the number
of intersections of $\Tilde{\pi}$ with $B^{\leq t}$ strictly
decreases, by the number of full turns made by the trajectory. Thus
the diagram $(B^{\leq t},\pi)$ can be made tighter in any case.
\end{proof}

\begin{theorem}
\label{t:konig-ext} Let $A\in\mmset^{(d+1)\times d}$ and let $t$ be
the greatest number $h$ such that $A^{(h)}$ has a $d\times d$
submatrix with nonzero permanent. Then for this value $t$ there is a
tight diagram $(B^{\leq t},\pi)$, such that all entries of $B^{\leq
t}$ are not greater than $t$, and all entries of $\pi$ are not
smaller than $t$.
\end{theorem}
\begin{proof}
The K\"{o}nig theorem (by the discussion in the beginning of this
section) yields a diagram $(B^{\leq t},\pi)$ which is not
necessarily tight. However, a tight diagram can be obtained from it
by repeated application of Lemma~\ref{l:improve}.
\end{proof}

\begin{proof}[Proof of Theorem~\ref{t:intsep}]
We will prove the following claim {\bf by induction}:\\
If $A\in\mmset^{(d+1)\times d}$ (with finite entries) contains a
permutation $\pi$ such that $a_{i\pi(i)}\geq t$ for all $i$ (except
for $i$ being the free row $f$), then there is a point $z$ with all
coordinates not less than $t$, which internally separates the rows
of $A$.

The case $d=1$ is the {\bf basis} of induction. In this case $A$
consists of just two numbers, say $x$ and $y$, and we can take
$z=\max(x,y)$ as the ``separating point''. Then one of the numbers
belongs to the sector $\{s\mid s\leq z\}$, and the remaining one to
$\{s\mid s\geq z\}$.

We now assume that the claim holds for all $d<n$, and let
$A\in\mmset^{(n+1)\times n}$ have only finite entries. By
Theorem~\ref{t:konig-ext}, there is a permutation $\pi$, a free
index $f$ such that $a_{i\pi(i)}\geq t$ for all $i\neq f$, and a
submatrix $B^{\leq t}:=A_{M_1N_1}$ with $a_{ij}\leq t$ for $i\in
M_1,j\in N_1$ such that the diagram $(B^{\leq t},\pi)$ is tight. Let
$M_2$ and $N_2$ be the complements of $M_1$ in $\{1,\ldots,n+1\}$
and of $N_1$ in $\{1,\ldots,n\}$, respectively. As the diagram is
tight, for each column with an index in $N_2$ the corresponding
entry of $\pi$ is in $A_{M_1N_2}$. Let $\Tilde{M}_1$ be the set of
rows consisting of the free row (which belongs to $M_1$ since the
diagram is tight), and the rows of $M_1$ such that $\pi(i)\in N_2$,
see Figure~\ref{f:tight}. Then the number of elements in $N_2$ is
one less than that of $\Tilde{M}_1$, and the matrix
$A_{\Tilde{M}_1N_2}$ contains a permutation $\pi'$ induced by $\pi$,
with all entries not smaller than $t$. Let $n'$ be the number of
elements in $N_2$, so $n'<n$. By the induction hypothesis there
exists an $n'$-component vector $z$ internally separating the rows
of $A_{\Tilde{M}_1N_2}$.

Define $x$ by $x_i=z_i$ for $i\in N_2$ and $x_i=t$ for $i\in N_1$.
We claim that $x$ is the separating point. Since the diagram is
tight, we have $\pi(i)\in N_1$ for all $i\in M_2$, and we also have
$\pi(i)\in N_1$ for all $i\in M_1\bez\Tilde{M}_1$ by the definition
of $\Tilde{M_1}$. This implies that $x$ satisfies $a_{i\pi(i)}\geq
t$ for all $i\notin\Tilde{M_1}$, determining the sectors in which
the rows with these indices lie. The sectors for the rows with
indices in $\Tilde{M}_1$ are determined by $z$ (i.e., by induction),
also using that $a_{ij}\leq t$ for all $i\in\Tilde{M}_1$ and $j\in
N_1$.
\end{proof}

\section{An application of topological Radon theorem}\label{s:radon-helly}

In this section we go beyond the max-min semiring considering what
we call the {\em max-T} semiring $\Tmax$: this is the unit interval
$\mmset=[0,1]$ equipped with the tropical addition $a\oplus
b:=\max(a,b)$ and multiplication $\otimes_T$ played by a $T$-norm
$T\colon \mmset\times\mmset\to\mmset$. These operations were
introduced in \cite{sklar} and a standard reference is the monograph
\cite{pap}.

\begin{definition} A triangular norm (briefly $T$-norm) is a binary operation $T$ on the unit interval $[0, 1]$ which
is 
associative, monotone and has $1$ as neutral element, i.e., it is a function
$T : [0,1]^2\to [0,1]$ such that for all $x,y,z\in [0,1]$:
\begin{enumerate}
\item [(T1)] $T(x,T(y,z))=T(T(x,y),z)$,
\item [(T2)] $T(x,y)\le T(x,z)$ and $T(y,x)\le T(z,x)$ whenever $y\le z$,
\item [(T3)] $T(x,1)=T(1,x)=x$.
\end{enumerate}

A $T$-norm is continuous if for all convergent sequences $(x_n)_n, (y_n)_n\in [0,1]^{\mathbb{N}}$ we have
$$
\lim_{n\to \infty}T(x_n,y_n)=T(\lim_{n\to \infty} x_n, \lim_{n\to \infty} y_n).
$$
\end{definition}

\begin{remark}
The axioms of semiring also require $0$ to be absorbing with
respect to multiplication, that is, $T(x,0)=T(0,x)=0$.
Note that this law follows from (T2,T3) and since $1$
is the greatest element.
\end{remark}

The multiplication $\otimes_T$ can be any of the continuous T-norms known in the fuzzy sets theory, including the usual multiplication, $\otimes=\min$ which we studied above, and the {\L}ukasiewicz T-norm $a\otimes_{\text{\L}} b:=\max(0,a+b-1)$.

Note that the case of
usual multiplication yields a part of the max-times semiring, isomorphic to the
non-positive part of the tropical/max-plus semiring.

Below we consider $\mmset^d$, the set of $d$-vectors with components in $\mmset$,  equipped
with the componentwise tropical addition and T-multiplication by scalars.
A set $C\subseteq\mmset^d$ is called {\em max-T convex} if, together with any $x,y\in C$, it contains all combinations $\lambda\otimes_T x\oplus\mu\otimes_T y$ where $\lambda\oplus\mu=1$.

For any set  $X\subseteq\mmset^d$, the {\em max-T convex hull}
of $X$ is defined as the smallest max-T convex set containing $X$. Using the axioms of
semiring, or 1)-4) above, it can be shown that the max-T convex hull of
$X$ is the set of all {\em max-T convex combinations}
$$
\bigoplus_{i=1}^m \lambda_i\otimes_T x^i\colon m\geq 1,\ \bigoplus_{i=1}^m \lambda_i=1,
$$
of all $m$-tuples of elements $x^1,\ldots,x^m\in X$. The max-T convex hull of a finite set of
points is also called a {\em max-T convex polytope}.

\if{
We introduce some standard terminology. If $x^0,x^1,\dots,x^k$ are points in $\R^m$ such that $\{x^1-x^0, x^2-x^0, \dots, x^k-x^0\}$ is
a linearly independent set in $\R^m$, then we say that these points
are \emph{affinely independent}. Let $0\le k\le m$, and $x^0, x^1, \dots , x^k$ be affinely
independent points in $R^m$. The $k$-simplex $\Delta = (x^0, x^1, \dots , x^k)$ is the
following subset of $\R^m$:
\begin{equation*}
\Delta=\left \{x\in \R^m\vert x=\sum_{i=0}^k\mu_ix^i, \sum_{i=0}^k \mu_i=1, 0\le \mu_i\le 1\right \}.
\end{equation*}

Since the points $x^0, x^1,\dots , x^k$ are affinely independent, the reals $\mu_i,0\le i\le k,$ are uniquely determined by $x$. The points $x^0,\dots, x^k$ are the \emph{vertices of $\Delta$}, and $k$ is the \emph{dimension} of $\Delta$. A simplex $\Delta_1$ is a \emph{face} (\emph{proper face}) of a simplex $\Delta_2$ if the vertex-set of $\Delta_1$ is
a subset (proper subset) of the vertex-set of $\Delta_2$.


By the unit simplex of dimension $k$ we mean the set
\begin{equation*}
\Delta_0^k=\left \{(\mu_0,\mu_1,\dots,\mu_k)\in \R^{k+1}\vert \sum_{i=0}^k \mu_i=1, 0\le \mu_i\le 1\right \},
\end{equation*}
is a $k$-simplex whose extremal vertices are the standard basis vectors in $\R^{d+1}$.
}\fi

We further make use of the following theorem of general topology that can be found in \cite{barany}. By the unit simplex of dimension $d$ we mean the set
\begin{equation*}
\Delta_d=\left \{(\mu_0,\mu_1,\dots,\mu_d)\in \R^{d+1}\vert \sum_{i=0}^d \mu_i=1, 0\le \mu_i\le 1\right \},
\end{equation*}
in the usual real space $\R^{d+1}$ and with the usual arithmetics.

\begin{theorem}[Topological Radon's theorem] If $f$ is any continuous function from
$\Delta_{d+1}$ to a $d$-dimensional linear space, then $\Delta_{d+1}$ has two disjoint faces whose images under $f$ are not disjoint.
\end{theorem}

\begin{theorem}[Radon's theorem for max-T] Let $X$ be a set of $d + 2$ points in $\mmset^d$. Then there are two pairwise disjoint subsets $X^1$ and $X^2$ of $X$ whose max-T convex hulls have a common point.
\end{theorem}

\begin{proof} Let $X=\{x^0,x^1,\dots,x^{d+1}\}\subseteq \Tmax^d$. We construct a continuous map $f$ from $\Delta_{d+1}$ to the max-T convex hull of $X$ that maps the faces of $\Delta_{d+1}$ into max-T convex hulls of subsets of $X$ and apply topological Radon's theorem to $f$.
Define
\begin{equation*}
\Delta_{d+1}^{\max}=\left \{ (\mu_0,\mu_1,\dots,\mu_{d+1})\in [0,1]^{d+2}\vert \max\{\mu_i, 0\le i\le d+1\}=1\right\}.
\end{equation*}

Using ordinary arithmetics, consider the map $\phi_1:\Delta_{d+1}^{\max}\to \Delta_{d+1}$ given by:
\begin{equation*}
\phi_1(\mu_0,\mu_1,\dots,\mu_{d+2})=\left ( \frac{\mu_0}{\sum_{i=0}^{d+1}\mu_i},  \frac{\mu_1}{\sum_{i=0}^{d+1}\mu_i}, \dots,  \frac{\mu_{d+1}}{\sum_{i=0}^{d+1}\mu_i} \right ),
\end{equation*}
which is clearly a homeomorphism, and thus has a continuous inverse. Moreover, for any subset of indices $I=\{i_1,i_2, \dots, i_k\}\subseteq \{0,1,2,\dots,d+1 \}$, $\phi_1$ maps the max-T convex hull of the standard vectors $e^{i_1}, \dots, e^{i_k}$ into the face of the simplex $\Delta_{d+1}$ determined by the vertices $e^{i_1}, \dots, e^{i_k}$.

Consider also the map $\phi_2$ defined on $\Delta_{d+1}^{\max}$ with values in $\mmset^d$ given  by:
\begin{equation*}
\phi_2(\mu_0,\mu_1,\dots,\mu_{d+1})=\max(\mu_0\otimes x^0,\mu_1\otimes x^1,\dots,\mu_{d+1}\otimes x^{d+1}),
\end{equation*}
which for any subset of indices $I$ as above takes the max-T convex hull of the standard vectors $e^{i_1}, \dots, e^{i_k}$ into the max-T convex hull of the vectors $x^{i_1}, \dots, x^{i_k}$.

Define now $f=\phi_2\circ \phi_1^{-1}$ on $\Delta_{d+1}$ with values in $\mmset^d$. Applying
to it the topological Radon theorem we get the claim.
\end{proof}

\begin{remark} It is of interest to find a purely combinatorial proof of max-min Radon's theorem, or in the case of other known T-norms.
\end{remark}

The following theorem is known more generally in abstract convexity, as
a consequence of Radon's theorem.

\begin{theorem}[Helly's theorem] Let $F$ be a finite collection of max-T convex sets in $\mmset^d$. If every $d + 1$ members of $F$ have a nonempty intersection, then the whole collection have a nonempty intersection.
\end{theorem}

\begin{proof} Let $C^1, \dots, C^n$ be max-T convex sets in $\mmset^d$ and suppose that whenever $d+1$ sets among them
are selected, they have a nonempty intersection. We proceed by induction on $n$. First assume that
$n = d + 2$. Define $x^i$ to be a point in the set $\cap^{d+2}_{j=1; j\not =i}C_j$. We have then $d + 2$ points $x^1, \dots, x^{d+2}$. If two of them
are equal, then this point is in the whole intersection. Hence, we can assume that all the $x^i$ are different.
By the Radon theorem, we have two disjoint subsets $S$ and $T$ partitioning $\{1,\dots,d + 2\}$ such that there
is a point $x$ in $\conv(\cup_{i\in S}x^i)\cap \conv(\cup_{i\in T}x^i)$. This point $x$ belongs to every $C^i$.
Indeed, take $j\in\{1,\dots,d + 2\}$, which is either in $S$ or in $T$. Suppose without loss of generality that $j\in S$.
Then, $\conv(\cup_{i\in T}x^i)$ is included in $C^j$ , and so $x \in C^j$ . The case $n = d + 2$ is proved.

Suppose now that $n > d+2$ and that the theorem is proved up to $n-1$. Define $C'^{n-1} := C^{n-1}\cap C^n$. When $d + 2$ convex sets $C^i$ are selected, they have a nonempty intersection, according to what we have just proved. Hence, every $d+1$ members of the collection $C^1, \dots, C^{n-2}, C^{n-1}$ have a nonempty intersection.

By induction, the whole collection has a nonempty intersection.
\end{proof}

The following two theorems are also known more generally in
abstract convexity, as a consequences of Helly's theorem.

\begin{theorem}[Centerpoint theorem] Let $P$ be a collection of $n$ points in $\mmset^d$. Then there exists a point $p\in \mmset^d$
(the \emph{centerpoint}) such that every max-T convex set containing
more than $dn/(d+1)$ points of $P$ also contains $p$.
\end{theorem}

\begin{proof} First construct all max-T convex polytopes containing more then $dn/(d+1)$ points in $P$.
Any point lying in all such polytopes is the required point.
Consider a $(d + 1)$-tuple of such polytopes. The complement of each
polytope in the tuple contains less then $n/(d+1)$ points from
$P$. The union of all $(d+1)$ complements of the polytopes in the
tuple contains less then $n$ points from $P$. Thus the complement of
the union, which is the intersection of all polytopes, is nonempty.
We only have to prove that given a set of convex polytopes such that
every $(d + 1)$-tuple has a non-empty intersection, all of them have
a non-empty intersection. But this is Helly's theorem.
\end{proof}

As $\mmset^d=[0,1]^d$ is endowed with the usual Euclidean topology
we observe that a max-T convex set is compact if and only if it is
closed.

\begin{theorem}[Helly's theorem for infinite collections of convex sets]\label{helly-infinite} Suppose $F$ is an infinite, possibly uncountable family
of max-T convex and compact sets in $\mmset^d$. Suppose that every $d + 1$ of them have a nonempty intersection. Then the whole family has a non-empty intersection.
\end{theorem}

\begin{proof} Let $F=\{B_i\}_{i\in I}$. According to Helly's theorem, every finite collection of $B_i$'s has a nonempty intersection. Fix a member $K$ of $F$ and define $G_i=\complement B_i$, $i\in I$. Assume that no point of $K$ belongs to all $B_i$. Then the family $\{G_i\}_{ i\in I}$ form an open cover for the the compact set $K$. One can find a finite subcover $G_{i_1},\dots,G_{i_l}$ such that $K\subseteq G_{i_1}\cup \dots \cup G_{i_l}$. But this means $K\cap B_{i_1}\cap \dots \cup B_{i_l}=\emptyset$, a contradiction.
\end{proof}

Let us conclude this section with Tverberg's theorem for max-T,
which can be derived from the more general topological version.

\begin{conjecture}[Topological Tverberg's theorem] If $f$ is any continuous function from
$\Delta_{(d+1)(r-1)}$ to a $d$-dimensional linear space, then $\Delta_{(d+1)(r-1)}$ has $r$ disjoint faces whose images under $f$ contain a common point.
\end{conjecture}

\begin{conjecture}[Tverberg's theorem for max-T]
Let $X$ be a set of $(d+1)(r-1) + 1$ points in
$\mmset^d$. Then there are $r$ disjoint subsets $X^1,\ldots,X^r$
of $X$ whose max-T convex hulls have a common point.
\end{conjecture}

It is known that the topological Tverberg's theorem is true for $d\geq 1$ and $r$ equal to a prime number~\cite{BSS-81}, and moreover for $d\geq 1$ and $r$ equal to a power of a prime~\cite{Vol-96}. By the above argument, it also shows Tverberg's theorem in max-T for these cases.


\begin{thebibliography}{10}

\bibitem{barany} E.~G.~Bajm\'oczy and I. B\' ar\' any.  A common generalization of Borsuk's and Radon's theorem. Acta Mathematica Hungarica \textbf{34} (1979) 347--350.

\bibitem{BSS-81} I. B\' ar\' any, S.~B.~Shlosman and A.~Sz\"{u}ks.  On a topological
generalization of a theorem of Tverberg. J. Lond. Math. Soc. \textbf{23} (1981) 158--161.

\bibitem{BCS-87} P,~Butkovi\v{c}, K.~Cechl\'arov\'a and P.~Szabo. Strong linear independence in bottleneck algebra.
Linear Algebra Appl., \textbf{94} (1987) 133-155.


\bibitem{CGQS-05}
G.~Cohen, S.~Gaubert, J.P. Quadrat, and I.~Singer. Max-plus convex sets and functions. In G.~Litvinov and V.~Maslov, editors, {\em Idempotent Mathematics
and Mathematical Physics}, volume 377 of {\em Contemporary Mathematics}, pages 105--129. AMS, Providence, 2005.
\newblock E-print arXiv:math/0308166.

\bibitem{develin-etc} M.~Develin, F.~Santos, B.~Sturmfels. On the rank of a tropical matrix. In "Discrete and Computational
Geometry" (E. Goodman, J. Pach and E. Welzl, eds), MSRI Publications, Cambridge Univ. Press, 2005, 213--242.

\bibitem{Gav-01}
M.~Gavalec. Solvability and unique solvability of max-min fuzzy
equations. Fuzzy Sets and Systems \textbf{124} (2001) 385-393.

\bibitem{Gav:04}
M.~Gavalec. Periodicity in Extremal Algebra. Gaudeamus, Hradec Kr\'alov\'e, 2004.

\bibitem{G-Meu}
S.~Gaubert and F. Meunier.
Carath\'eodory, Helly and the Others in the Max-Plus World.
Discrete and Computational Geometry, \textbf{43}, (2010) 648--662.


\bibitem{EJN} J.~Eskeldson, M.Jaffe, V. Nitica. A metric on max-min algebra, Contemporary Mathematics, this volume, AMS, Providence.

\bibitem{pap} E.P.~Klement, R.~Mesiar, E.~Pap, Triangular Norms, Kluwer Academic Publishers, Dordrecht, 2000.

\bibitem{LMS-01} G.~L.~Litvinov, V.~P.~Maslov, G.~B.~Shpiz, Idempotent functional analysis: An Algebraic Approach,
Math Notes \textbf{69} (2001), 758--797.

\bibitem{Nit-09}
V.~Nitica. The structure of max-min hyperplanes. Linear Algebra Appl. \textbf{432} (2010), 402–-429.

\bibitem{N-Ser1}
V.~Nitica and S.~Sergeev. On semispaces and hyperplanes in max-min convex geometry. Kybernetika \textbf{46} (2010), 548--557.

\bibitem{N-Ser2}
V.~Nitica and S.~Sergeev. An interval version of separation by semispaces in max-min convexity. Linear Algebra Appl. \textbf{435} (2011), 1637–-1648.

\bibitem{NS-1} V.~Nitica and I.~Singer. Max-plus convex sets and max-plus semispaces I. Optimization \textbf{56} (2007) 171--205.

\bibitem{NS-08I}
V.~Nitica and I.~Singer. Contributions to max-min convex geometry. I. Segments. Linear Algebra Appl. \textbf{428} (2008), 1439--1459.

\bibitem{NS-08II}
V.~Nitica and I.~Singer. Contributions to max-min convex geometry. II. Semispaces and convex sets.
Linear Algebra Appl. \textbf{428} (2008), 2085--2115.

\bibitem{sklar} B.~Schweizer and A.~Sklar, Probabilistic Metric Spaces, North-Holland, New York, 1983

\bibitem{Ser-03}
S.~N.~Sergeev. Algorithmic complexity of a problem of idempotent convex geometry. Math. Notes (Moscow), \textbf{74} (2003), 848--852.

\bibitem{Vol-96}
A.~Yu.~Volovikov. On a topological generalization of the Tverberg theorem.
Math. Notes (Moscow), \textbf{59} (1996), 324-326.

\bibitem{Zim-77}
K.~Zimmermann. A general separation theorem in extremal algebras. Ekonom.-Mat. Obzor (Prague), \textbf{13} (1977), 179--201.

\bibitem{Zim-81} K.~Zimmermann. Convexity in semimodules. Ekonom.-Mat. Obzor (Prague), \textbf{17} (1981), 199--213.


\end{thebibliography}
\end{document}